%% file: IsotonicCI2.tex
\documentclass[aos,preprint]{imsart}
\RequirePackage[OT1]{fontenc}
\RequirePackage{amsthm,amsmath}
\RequirePackage[numbers]{natbib}
\RequirePackage{hypernat}
\usepackage{pictexwd}
\usepackage{amssymb,graphicx,url,bbm}
\usepackage[colorlinks,citecolor=blue,urlcolor=blue]{hyperref}

\startlocaldefs

\input commands3.tex

\numberwithin{equation}{section}
\theoremstyle{plain}
\endlocaldefs

\usepackage{amsmath,here,amsfonts,latexsym,sym1,graphicx,here}
\usepackage{amsthm,curves}
\usepackage{graphicx}
\usepackage{caption}
\usepackage{subcaption}
\usepackage{float}

\advance \topmargin by -\headheight \advance \topmargin by
-\headsep

\def\m{\mu}

\def\P{{\mathbb P}}
\def\F{{\mathbb F}}
\def\G{{\mathbb G}}

\begin{document}

\begin{frontmatter}
\title{Nonparametric confidence intervals for monotone functions}
\runtitle{Nonparametric isotonic confidence intervals}

\begin{aug}
\author{\fnms{Piet} \snm{Groeneboom}\corref{}\ead[label=e1]{P.Groeneboom@tudelft.nl}
\ead[label=u1,url]{http://dutiosc.twi.tudelft.nl/\textasciitilde pietg/}}
and\ \
\author{\fnms{Geurt} \snm{Jongbloed}\ead[label=e2]{G.Jongbloed@tudelft.nl}
\ead[label=u2,url]{http://dutiosc.twi.tudelft.nl/\textasciitilde geurt/}}
\affiliation{Delft University of Technology}
\runauthor{P.\ Groeneboom and G.\ Jongbloed}
\end{aug}

\begin{abstract}
We study nonparametric isotonic confidence intervals for monotone functions. In \cite{mouli_jon:01} pointwise confidence intervals, based on likelihood ratio tests using the restricted and unrestricted MLE in the current status model, are introduced. We extend the method to the treatment of other models with monotone functions, and demonstrate our method by a new proof of the results in \cite{mouli_jon:01} and also by constructing confidence intervals for monotone densities, for which still theory had to be developed. For the latter model we prove that the limit distribution of the LR test under the null hypothesis is the same as in the current status model. We compare the confidence intervals, so obtained, with confidence intervals using the smoothed maximum likelihood estimator (SMLE), using bootstrap methods. The ``Lagrange-modified'' cusum diagrams, developed here, are an essential tool both for the computation of the restricted MLEs and for the development of the theory for the confidence intervals, based on the LR tests.
\end{abstract}



\end{frontmatter}

\section{Introduction}
\label{section:intro}
\setcounter{equation}{0}
In many situations one would like to estimate functions under the condition that they are monotone. Apart from giving algorithms for computing such estimates and from deriving their (usually asymptotic) distribution theory, it is also important to construct confidence intervals. These intervals can be uniform (in which case they are usually called confidence bands) as well as pointwise.

In this paper we consider two methods to obtain pointwise confidence intervals for distribution functions and monotone densities, based on nonparametric estimators. One approach, that of a (nonparametric) likelihood ratio (LR) test, based on the maximum likelihood estimator (MLE) in the model, is related to the one taken in \cite{mouli_jon:01} and \cite{mouli_jon:05}. The other approach, using a smoothed maximum likelihood estimator (SMLE) is based on an estimator introduced in \cite{piet_geurt_birgit:10} and further analyzed in \cite{piet_geurt:14}. Our methods can also be applied to monotone nonparametric least squares estimates of monotone regression functions. 

There are some important differences between the approaches, based on the MLE and SMLE, respectively. How appropriate it is to use the MLE will largely depend on whether one expects (or allows) that the underlying monotone function will have jumps.
Secondly, the bias of the MLE does not play a role in the construction of the confidence intervals based on the MLE. But if one constructs confidence intervals, using the SMLE with an optimal bandwidth, the bias will not be negligible in the limiting distribution. There is an extensive literature on how to deal with the bias in nonparametric function estimation, some approaches use undersmoothing, other approaches oversmoothing. A recent paper, discussing this literature and giving a solution for confidence bands, is \cite{Hall_Horowitz:13}. We will use undersmoothing, as suggested in \cite{hall:92}.

The method of constructing confidence intervals, based on the likelihood ratio test for the MLE, and the method using the SMLE are both asymptotically pivotal. For the method, based on the likelihood ratio test for the MLE, this arises from the universality properties of likelihood ratio tests. For the intervals, based on the SMLE, this is based on using bootstrap intervals for a ``studentized" statistic, together with the undersmoothing. We now first describe two models that will be studied thoroughly in this paper.

\begin{example} (Monotone density functions)
\label{example:monotone_dens}
{\rm The classical example of a monotone estimate of a monotone function is the so-called {\it Grenander estimator}. Let $X_1,\dots,X_n$ be a sample of random variables, generated by a decreasing density $f_0$ on $[0,\infty)$. The MLE $\hat f_n$ of $f_0$ is the Grenander estimator, which is by definition the left derivative of the least concave majorant of the empirical distribution function $\F_n$ of $X_1,\dots,X_n$, as proved in \cite{Grenander:56} (see also Lemma 2.2 in \cite{piet_geurt:14}). This is also the first example in \cite{mouli_jon:01}, where there is the (implicit) conjecture that pointwise confidence intervals, based on the Grenander estimate, will have similar properties as the confidence intervals for the current status model (see the next example), based on a likelihood ratio test for the MLE. The difficulty in proving this result for the monotone density model resides in the constraint that the density integrates to $1$, a condition which does not play a role in constructing LR tests for the current status model. We shall prove that the conjecture in \cite{mouli_jon:01} is correct and that one can use the same critical values as in the current status model in the construction of the asymptotic confidence intervals. We  also compare the confidence intervals, obtained in this way, with confidence intervals, based on the SMLE, using bootstrap methods and asymptotic normality of the SMLE.
}
\end{example}

\begin{example} (The current status model)
\label{example:curstat_model}
{\rm Consider a sample $X_1,X_2,\ldots,X_n$, drawn from a distribution with distribution function $F_0$. Instead of observing the $X_i$'s (which can be thought of as an event time, such as `getting infected'), one only observes for each $i$ whether or not $X_i\le T_i$ for some random (inspection time) $T_i$ (independently of the other $T_j$'s and all $X_j$'s). More formally, instead of observing $X_i$'s, one observes
\begin{equation}
\label{eq:obsIC1}
(T_i,\Delta_i)=(T_i,1_{[X_i\le T_i]}),\,\,1\le i\le n.
\end{equation}
One could say that the $i$-th observation represents the {\it current status} of item $i$ at time $T_i$.

The problem is to estimate the unknown distribution function $F$ based on the data given in (\ref{eq:obsIC1}). Denote the ordered realized $T_i$'s by $t_1<t_2<\ldots<t_n$ and the associated realized values of the $\Delta_i$'s by $\delta_1,\ldots,\delta_n$. For this problem the log likelihood function in $F$ (conditional on the $T_i$'s) is given by
\begin{equation}
\label{eq:loglikICI}
\ell(F)=\sum_{i=1}^n \left\{\delta_i\log F(t_i)+(1-\delta_i)\log(1-F(t_i))\right\}.
\end{equation}

The MLE maximizes $\ell$ over the class of {\it all} distribution functions. Since distribution functions are by definition nondecreasing, the problem belongs to the class of problems we want to study. As can be seen from the structure of (\ref{eq:loglikICI}), the value of $\ell$ only depends on the values that $F$ takes at the observed time points $t_i$; the values of $F$ in between are not relevant as long as $F$ is nondecreasing. Hence one can choose to consider only distribution functions that are constant between successive observed time points $t_i$. Lemma \ref{lem:charMLECS} below shows that this estimator can be characterized in terms of a greatest convex minorant of a certain diagram of points.

The main result of \cite{mouli_jon:01} is that confidence intervals, based on an LR test for the MLE, can be constructed, and that this is a pivotal way of constructing asymptotic confidence intervals, since the limit distribution does not depend on the parameters (under certain conditions). We will give a new proof, which is in line with our proof for the monotone density model.
}
\end{example}

There are numerous other models where our approach can be adopted. Examples include the model where one has a monotone hazard rate and right censored observations (see Sections 2.6 and 11.6 in \cite{piet_geurt:14}), the competing risk model with current status observations (see \cite{GroeneboomMaathuisWellner08a}) and monotone regression.

The methods based on the LR tests for the MLEs in the context of Example \ref{example:monotone_dens} and \ref{example:curstat_model} follow the same line of argument, where, in both cases, an essential role is played by the penalization parameter $\hat\m_n$, which is of order $O_p(n^{-2/3})$.  Our methods rely on cumulative sum (cusum) diagrams which could be called `Lagrange-modified' cusum diagrams, since they incorporate the Lagrange multipliers for the penalties. Asymptotic distribution theory is derived from the asymptotic properties of the Lagrange multipliers, used to construct these cusum diagrams. Once this has been done, the theory for the confidence intervals follows.

\section{Confidence intervals for the current status model}
\label{section:CI_current_status}
\setcounter{equation}{0}

The following lemma characterizes the unrestricted MLE in the current status model. This is Example \ref{example:curstat_model} in Section \ref{section:intro}, and we use the notation introduced there.

\begin{lemma}[Lemma 2.7 in \cite{piet_geurt:14}]
\label{lem:charMLECS}
Consider the cumulative sum diagram consisting of the points $P_0=(0,0)$ and
\begin{equation}
\label{cusum_curstat}
P_i=\left(i,\sum_{j=1}^i\delta_j\right),\,\,\,1\le i\le n,
\end{equation}
recalling that the $\delta_i$'s correspond to the $t_i$'s, which are sorted.
Then the unrestricted MLE $\hat F_n$ is given at the point $t_i$ by the left derivative of the greatest convex minorant of this diagram of points, evaluated at the point $i$. This maximizer is unique among all subdistribution functions with mass concentrated on the inspection times $t_1,\ldots,t_n$.
\end{lemma}

\begin{remark}
\label{remark:leftcont_curstat}
{\rm
The {\it left} derivative of the convex minorant at $P_i$ determines the value of $\hat{F}_n$ at $t_i$ and hence (by right continuity of the step function) on $[t_i,t_{i+1})$, a region to the {\it right} of $t_i$.
}
\end{remark}

The characterization via Lemma \ref{lem:charMLECS} is well-known and a proof can e.g.\ be found in \cite{rwd:88} and \cite{piet_geurt:14}.

For the confidence intervals based on likelihood ratio tests for the MLE, we also have to compute the MLE under the restriction that its value is equal to a prescribed value $a$ at a point $t_0$. There are different ways to do this. It is suggested in \cite{mouli_jon:01} to compute the restricted MLE in two steps. The restricted MLE $\hat F_n^{(0)}$ is computed for values at points $t$ to the left of $t_0$ under the restriction that $\hat F_n^{(0)}(t)\le a$ and for values at points $t$ to the right of $t_0$ under the restriction that $\hat F_n^{(0)}(t)\ge a$. To this end two cusum diagrams of type (\ref{cusum_curstat}) are formed. Let $m$ be such that $t_m\le t_0\le t_{m+1}$. Then, a diagram of type (\ref{cusum_curstat}) is formed, with $n$ replaced by $m$, for the values to the left of $t_0$.
Next the minimum of $a$ and the left derivative of the greatest convex minorant of this diagram of points is taken as the solution to the left of $t_0$. For the points on the right side of $t_0$  the cusum diagram consisting of the points
\begin{equation}
\label{cusum_curstat2}
P_0=(0,0) \mbox{ and }P_i=\left(i,\sum_{j=1}^i\left(1-\delta_{n-j+1}\right)\right),\,\,\,1\le i\le n-m,
\end{equation}
is considered and the maximum of $a$ and $1$ minus the left derivatives of the greatest convex minorant of this diagram of points, with the obvious renumbering, is taken as the solution $\hat F_n^{(0)}(t_i)$ to the right of $t_0$. Note that in this approach there is not necessarily a point $t_i$ where $\hat F_n^{(0)}(t_i)=a$ is actually achieved; we only have inequalities. Of course, in view of the log likelihood, allowing an extra jump of the distribution function at $t_0$, the value of $\hat F_n^{(0)}(t_0)$ can be taken equal to $a$ if this is required.

In view of our general approach, where we also will prove the result for monotone densities, we will follow a different path, where we make the connection with the penalization methods, studied in, e.g., \cite{grojosmooth:13} and earlier in \cite{WoodSun:93}.
We have the following result.

\begin{lemma}
\label{lemma:one_lambda_curstat}
Let $0<a<1$ and $1<i_0<n$ be such that $\delta_i=1$ for some $i\le i_0$ and $\delta_i=0$ for some $i>i_0$.
Moreover, let $t_0\in(t_{i_0},t_{i_0+1})$. Denote by $(\hat F_1,\dots,\hat F(t_0),\dots,\hat F_n)$ the vector of values of a piecewise constant nondecreasing function $\hat F$ at the observation points and at the point $t_0$, where $\hat F_{i_0}\le \hat F(t_0)\le \hat F_{i_0+1}$.  Then
\begin{enumerate}
\item[(i)]
If $\hat F_i$ is given by the left-hand slope of the greatest convex minorant of the cusum diagram with points $(0,0)$ and
\begin{equation}
\label{fenchel_null}
\left(i,\sum_{j=1}^i\d_j\right),\qquad i=1,\dots,n,
\end{equation}
and if $\hat F_{i_0}\le a\le \hat F_{i_0+1}$, we put $\hat F(t_0)=a$, and $\hat F^{(0)}=(\hat F_1,\dots,\hat F(t_0),\dots,\hat F_n)$ is the maximizer of $\sum_{i=1}^n\{\d_i\log F_i+(1-\d_i)\log(1-F_i)\}$, under the side condition $F(t_0)=a$.
\item[(ii)]
If $(\hat F_1,\dots,\hat F_n)$ is defined as in (i), but $\hat F_{i_0}>a$ or $\hat F_{i_0+1}<a$, we define $\hat\m\in\R$ to be the solution (in $\m$) of the equation
\begin{equation}
\label{equation_mu}
\max_{k\le i_0}\min_{i\ge i_0}\frac{\sum_{j=k}^i\d_j+n\m\,a(1-a)}{i-k+1}=a.
\end{equation}
and define $\hat F_i^{(0)}$ by the left-hand slope of the greatest convex minorant of the cusum diagram with points $(0,0)$ and
\begin{equation}
\label{fenchel_lambdas2}
\left(i,\sum_{j=1}^i\left\{\d_j+n\hat\m\,a(1-a)1_{\{j=i_0\}}\right\}\right),
\qquad i=1,\dots,n.
\end{equation}
We put $\hat F^{(0)}(t_0)=a$. Then $\hat F^{(0)}=(\hat F_1^{(0)},\dots,\hat F^{(0)}(t_0),\dots,\hat F^{(0)}_n)$ is the maximizer of\\ $\sum_{i=1}^n\{\d_i\log F_i+(1-\d_i)\log(1-F_i)\}$, under the side condition $F(t_0)=a$.
\end{enumerate}
\end{lemma}

\begin{remark}
{\rm The condition $\d_i=0$ for some $i>i_0$ is to avoid trivialities for the case that $\d_i=1$ for all $i\ge i_0$, in which case the only reasonable value of $F_i$ is 1 for $i\ge i_0$. A similar remark holds for the condition that $\d_i=1$ for some $i\le i_0$. If this were not the case, we would put $F_i$ equal to $0$ for $i\le i_0$. For the asymptotic confidence intervals we concentrate on {\it interior} points of the support of the distribution $F_0$.
}
\end{remark}

\begin{figure}[!ht]
\begin{subfigure}[b]{0.45\textwidth}
\includegraphics[width=\textwidth]{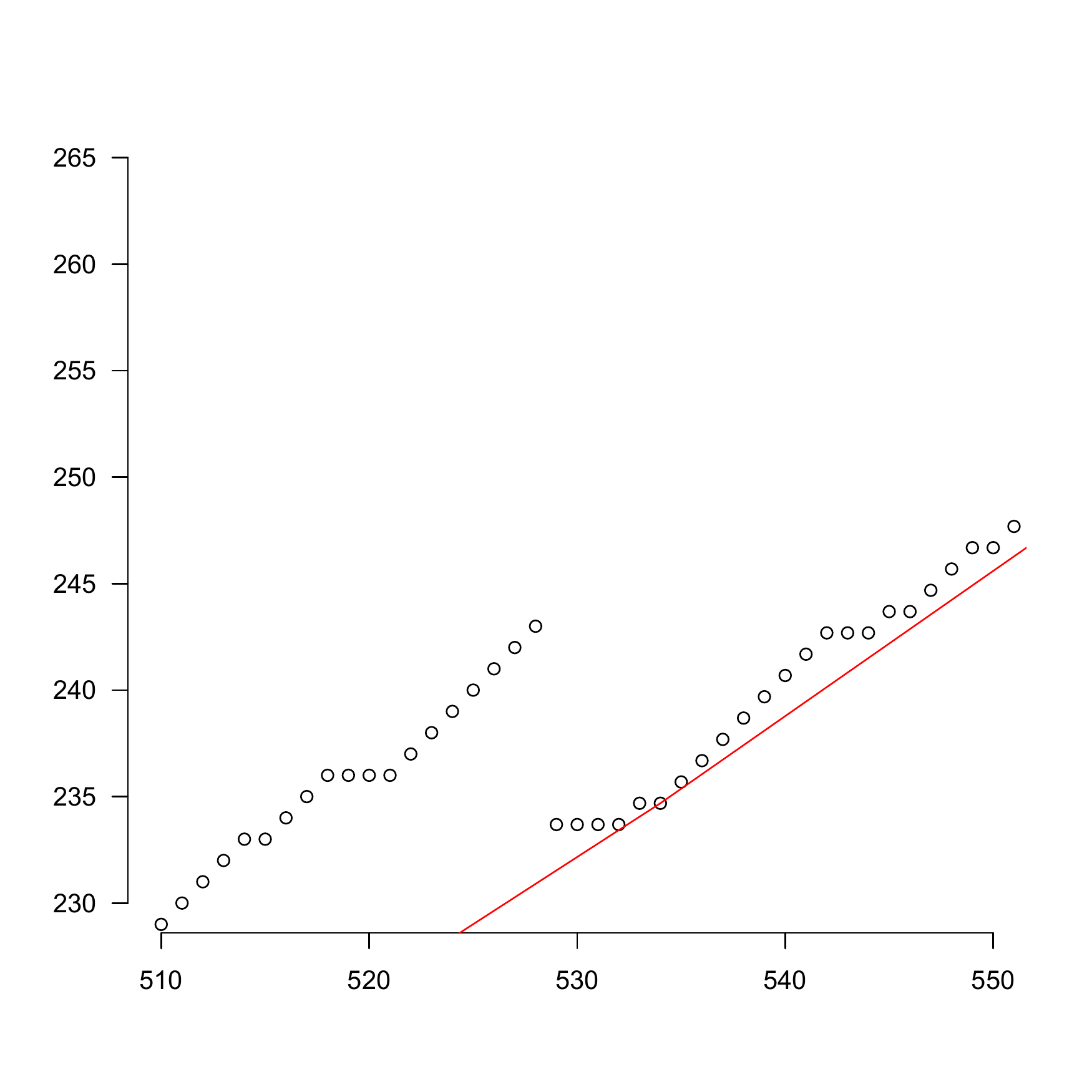}
\caption{}
\label{fig:curstat_cusum1}
\end{subfigure}
\begin{subfigure}[b]{0.45\textwidth}
\includegraphics[width=\textwidth]{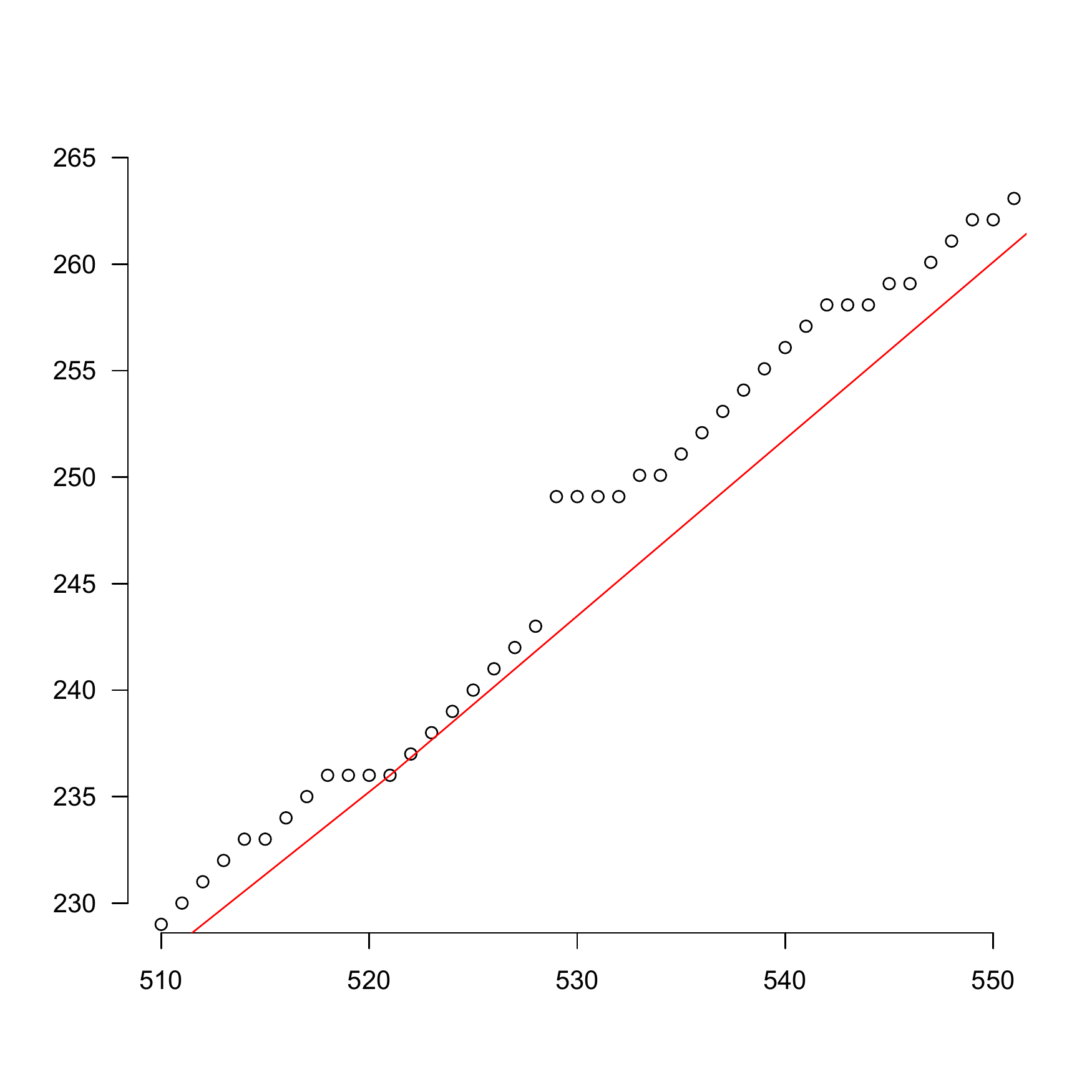}
\caption{}
\label{fig:curstat_cusum2}
\end{subfigure}
\caption{Pieces of two ``Lagrange modified'' cusum diagrams for the current status model, for sample size $1000$ from the truncated exponential distribution function $F_0$ on $[0,2]$; the observation distribution is uniform on $[0,2]$, and $t_0=1$, $F_0(t_0)=0.731058$. The unrestricted MLE $\hat F$ has value $0.722892$ at $t_0$. In the example we have: $t_{529}<t_0=1<t_{530}$.\\
Here (a) gives the local cusum diagram for $\hat F^{(0)}(t_0)=a=F_0(t_0)-0.1=0.631058$, $\hat\m=-0.039998$, and (b) gives the cusum diagram for $\hat F^{(0)}(t_0)=a=F_0(t_0)+0.1=0.831058$, $\hat\m=0.043355$. In both cases the big jump is at the point $i_0=529$.}
\label{fig:two_cusums100}
\end{figure}

\begin{proof}
(i) If  $\hat F_i$ is given by the left-hand slope of the greatest convex minorant of the cusum diagram (\ref{fenchel_null}), then $\hat F=(\hat F_1,\dots,\hat F_n)$ maximizes $\sum_{i=1}^n\{\d_i\log F_i+(1-\d_i)\log(1-F_i)\}$ without the side condition $\hat F(t_{i_0})\le a\le \hat F(t_{i_0+1})$. Since the side condition is also satisfied under (i), $\hat F$ is also the maximizer under this side condition in this case.\\
(ii) We can reduce the proof to the situation where $\d_n=0$. For if $\d_j=1$ for $j\ge i$, we put $F_j=1$ for $j\ge i$. For similar reasons we can assume $\d_1=1$. A similar reduction of the maximization problem was used in Proposition 1.3, p.\ 46 of \cite{GrWe:92}. An advantage of this reduction is that maximizing $\ell$ over  all vectors $F=(F_1,\ldots,F_n)$ with $0\le F_1\le \dots\le F_n\le1$ is equivalent to maximizing $\ell$ over the cone $C=\{F=(F_1,\dots,F_n):0\le F_1\le\dots\le F_n\}$.

Now first suppose $\hat F(t_{i_0})>a$ for the unrestricted solution in (i). Then we have to make $\hat F(t_{i_0})$ smaller to allow $\hat F(t_0)=a$. We do this by changing the object function to be maximized over $C$ into:
\begin{align}
\label{criterion_curstat}
\f_{\m}(F_1,\dots,F_n)&=\ell(F)+n\m\left(F_{i_0}-a\right)=\sum_{i=1}^n \left\{\d_i\log F_i+
\left(1-\d_i\right)\log\left(1-F_i\right)\right\}+n\m\left(F_{i_0}-a\right),
\end{align}
where $\m<0$ is a suitable Lagrange multiplier.

The elements of the cone $C$ can be uniquely expressed as positive linear combinations of its so-called generators
$$
g_1=(0,0,\dots,0,0,1),\,g_2=(0,0,\dots,0,1,1),\dots, g_n=(1,1,\dots,1,1,1).
$$
The necessary and sufficient Fenchel conditions for maximizing a concave function over a convex cone, (7.35) of \cite{piet_geurt:14}, applied to these generators, lead to the following inequalities:
\begin{align}
\label{fenchel1}
\left\langle\nabla\f_{\m}(F),g_j\right\rangle&=\sum_{j=i}^n \left\{\frac{\d_j-F_j}{F_j\{1-F_j\}}+n\m1_{\{j=i_0\}}\right\}\le0,\qquad i=1,\dots,n,
\end{align}
where $\nabla\f_{\m}(F)$ is the nabla vector $(\frac{\partial}{\partial F_1}\f_{\m},\dots,\frac{\partial}{\partial F_n}\f_{\m})$ at $F$ and $\m$ of the function (\ref{criterion_curstat}).
These inequalities can be rewritten as
\begin{align*}
\sum_{j=i}^n \left\{\frac{\d_j-F_j+n\m1_{\{j=i_0\}}a(1-a)}{F_j\{1-F_j\}}\right\}
\le0,\qquad j=1,\dots,n.
\end{align*}
We also have the equality part of the Fenchel conditions,
\begin{align}
\label{fenchel2}
\left\langle\nabla\f_{\m}(F),F\right\rangle=\sum_{j=1}^n\frac{\delta_j-F_j}{1-F_j}+n\m a=0.
\end{align}
Multiplying this relation on blocks of constancy of $F$ by $1-F_j$ (see the proof of Lemma 2.3 in \cite{piet_geurt:14}), we find:
\begin{align}
\label{fenchel2a}
\sum_{j=1}^n\bigl(\delta_j-F_j\bigr)+n\m a(1-a)=0.
\end{align}
The Fenchel conditions (\ref{fenchel1}) and (\ref{fenchel2}) or (\ref{fenchel2a}) are necessary and sufficient conditions for the MLE, restricted to be equal to $a$ at $t_{i_0}$.

It now follows that $\hat F^{(0)}$ is given by the left derivatives of the greatest convex minorant of the cusum diagram (\ref{fenchel_lambdas2}), where $\hat\m$ is the solution of the equation (\ref{equation_mu}). For the left derivative of the greatest convex minorant of the cusum diagram at $i_0$ is given by the left side of (\ref{equation_mu}), by a well-known maxmin characterization, see, e.g., Theorem 1.4.4 in \cite{rwd:88},  and if (\ref{equation_mu}) holds, we also have (\ref{fenchel2a}), since the greatest convex minorant will be equal to the second coordinate $\sum_{j=1}^n\{\d_j+\hat\m\,a(1-a)1_{\{j=i_0\}}\}$ of the cusum diagram at  $n$. Since $\hat F_{i_0}^{(0)}=a$, we can also let  $\hat F^{0}(t_0)=a$.

If $\hat F(t_{i_0+1})<a$, for $\hat F$ as in (i), we also have $\hat F(t_{i_0})<a$, and we reason in a similar way, this time for a Lagrange multiplier $\hat\m>0$. This will again give $\hat F^{(0)}_{i_0}=a$, and we can define $\hat F^{(0)}(t_0)=a$ again.
\end{proof}

\begin{remark}
\label{remark:error_jon_mouli}
{\rm
Cusum diagrams, incorporating the penalty, are shown in Figure \ref{fig:two_cusums100}. We have $\hat\m>0$ if $\hat F_{i_0+1}<a$ for the unrestricted solution of the maximization problem, and the cusum diagram for the restricted maximization problem is moved upward at $i_0$. If $\hat F_{i_0}>a$, it is the other way around. The penalties give a local deviation of the restricted MLE $\hat F^{(0)}$ from the unrestricted MLE, but outside a local neighborhood of the point of restriction, $\hat F^{(0)}$ and $\hat F$ will coincide again, where $\hat F^{(0)}$ picks up the same points of jump as $\hat F$.

Note, however, that we cannot say $\hat F^{(0)}(t)=a$ for the values $t$ where $\hat F^{(0)}(t)\ne \hat F(t)$. A typical picture is shown in Figure \ref{fig:restr_unrestrMLE}, where, on the region where $\hat F^{(0)}$ and $\hat F$ are different, the points of jump of $\hat F^{(0)}$ and $\hat F$ are at different locations. There is also not a ``contained in" relation in either direction for the sets of points of jump.
}
\end{remark}

\begin{figure}[!ht]
\begin{center}
\includegraphics[scale=0.5]{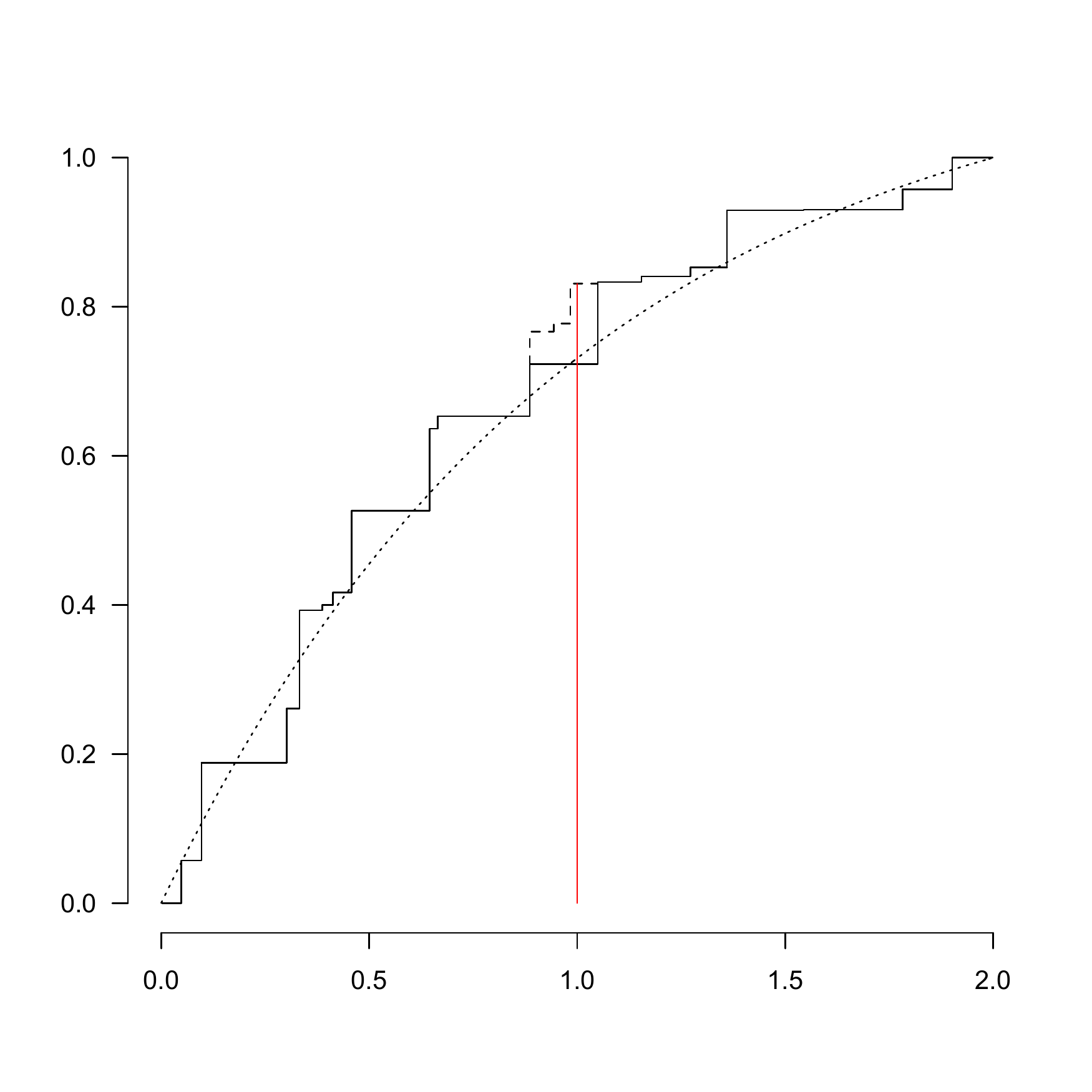}
\end{center}
\caption{The unrestricted MLE and restricted MLE for the same data as in Figure \ref{fig:two_cusums100}, where $F_0$ (dotted) is the truncated exponential on $[0,2]$ and the observation distribution is uniform on $[0,2]$. Moreover, $\hat F^{(0)}(1)=F_0(1)+0.1$. The deviation of the restricted MLE $\hat F^{(0)}$ from the unrestricted MLE $\hat F$ is dashed. The jumps of the restricted and unrestricted MLE do not coincide on the interval of deviation. The value of $\hat F_n^{(0)}$ at $t_0=1$ equals $0.831058$; the vertical bar connects the points $(1,0)$ and $(1,F_0(1)+0.1)$.}
\label{fig:restr_unrestrMLE}
\end{figure}

The proof of Theorem \ref{th:LR_current_status} below will use the following lemma, which is of a similar nature as results in \cite{grojosmooth:13}. To indicate the dependence on the sample size $n$, we now will denote the unrestricted and restricted MLE by $\hat F_n$ and $\hat F_n^{(0)}$, respectively.

\begin{lemma}
\label{lemma:order_mu_CS}
Under the conditions of Theorem \ref{th:LR_current_status} we have, if $a=F_0(t_0)$, as $n\to\infty$:
$$
\hat\m_n=O_p\left(n^{-2/3}\right).
$$
\end{lemma}

\begin{proof}
Consider the function
$$
\f(\m)=\max_{k\le i_0}\min_{i\ge i_0}\frac{\sum_{j=k}^i\d_j+n\m\,a(1-a)}{i-k+1},\qquad a=F_0(t_0).
$$
By the conditions of Theorem \ref{th:LR_current_status} we may assume that the observation times have two successive order statistics $T_{i_0}$ and $T_{i_0+1}$, as in Lemma \ref{lemma:one_lambda_curstat}, such that $t_0\in(T_{i_0},T_{i_0+1})$.
By the maxmin characterization of the unrestricted MLE $\hat F_n$, we have
$$
\f(0)=\max_{k\le i_0}\min_{i\ge i_0}\frac{\sum_{j=k}^i\d_j}{i-k+1}=\hat F_n(T_{i_0}).
$$
Let $k_1\le i_0$ and $i_1\ge i_0$ be the indices, satisfying
$$
\hat F_n(T_{i_0})=\frac{\sum_{j=k_1}^{i_1}\d_j}{i_1-k_1+1}=\max_{k\le i_0}\min_{i\ge i_0}\frac{\sum_{j=k}^i\d_j}{i-k+1}\,.
$$
Suppose $a>\hat F_n(T_{i_0})$ and let, for $\m>0$, $i_{\m}\ge i_0$ be the index such that
$$
\frac{\sum_{j=k_1}^{i_{\m}}\d_j+n\m a(1-a)}{i_{\m}-k_1+1}=\min_{i\ge i_0}\frac{\sum_{j=k_1}^{i}\d_j+n\m a(1-a)}{i-k_1+1}\,.
$$
Then, since the function
$$
\m\mapsto \min_{i\ge i_0}\frac{\sum_{j=k_1}^{i}\d_j+n\m a(1-a)}{i-k_1+1}
$$
is continuous and increasing in $\m$ and tends to $\infty$, as $\m\to\infty$, there exists a $\m>0$ such that
$$
\frac{\sum_{j=k_1}^{i_{\m}}\d_j+n\m a(1-a)}{i_{\m}-k_1+1}=\min_{i\ge i_0}\frac{\sum_{j=k_1}^{i}\d_j+n\m a(1-a)}{i-k_1+1}=a.
$$

Using $a=F_0(t_0)$ and denoting the empirical measure of $\{(T_j,\Delta_j)\,:\,1\le j\le n\}$ by $\P_n$, this means that
\begin{equation}
\label{mu_relation}
\m F_0(t_0)(1-F_0(t_0))=\int_{t\in[\t_-,T_{i_{\m}}]}\bigl\{F_0(t_0)-\d\bigr\}\,d\P_n(t,\d).
\end{equation}
where $\t_-=T_{k_1}$ is the last jump point of $\hat F_n$ before $t_{i_0}$.
By a well-known fact on the jump points of the MLE in the current status model (see, e.g., Lemma 5.4 and its proof on p. 95 of \cite{GrWe:92}), we have that $t_0-\t_-=O_p(n^{-1/3})$.
By the same type of argument, we can choose for each $\e>0$ an $M>0$ such that
$$
\P\left\{\int_{u\in[\t_-,t]}\bigl\{F_0(t_0)-\d\bigr\}\,d\P_n(u,\d)<0\right\}>1-\e,
$$
if $t>t_0+Mn^{-1/3}$. Denote the distribution function of the observation times by $G$, with corresponding empirical distribution function $\G_n$. Then, since we must have
\begin{align}
\label{random_versus_parabolic}
&0<\int_{t\in[\t_-,T_{i_\m}]}\bigl\{F_0(t_0)-\d\bigr\}\,d\P_n(t,\d)\nonumber\\
&=\int_{t\in[\t_-,T_{i_\m}]}\bigl\{F_0(t_0)-F_0(t)\bigr\}\,d\G_n(t)+\int_{t\in[\t_-,T_{i_\m}]}\bigl\{F_0(t)-\d\bigr\}\,d\P_n(t,\d)\nonumber\\
&=\int_{t\in[\t_-,T_{i_\m}]}\bigl\{F_0(t_0)-F_0(t)\bigr\}\,dG(t)
+\int_{t\in[\t_-,T_{i_\m}]}\bigl\{F_0(t_0)-F_0(t)\bigr\}\,d(\G_n-G)(t)\nonumber\\
&\qquad\qquad\qquad\qquad\qquad\qquad\qquad\qquad+\int_{t\in[\t_-,T_{i_\m}]}\bigl\{F_0(t)-\d\bigr\}\,d\P_n(t,\d)
\end{align}
by the positivity of $\m$, relation (\ref{mu_relation}) and the conditions of Theorem \ref{th:LR_current_status}, it now follows that $T_{i_{\m}}-t_0=O_p(n^{-1/3})$, and therefore
$$
\m F_0(t_0)(1-F_0(t_0))=\int_{t\in[\t_-,T_{i_{\m}}]}\bigl\{F_0(t_0)-\d\bigr\}\,d\P_n(t,\d)=O_p\left(n^{-2/3}\right),
$$
since $t_0-\t_-=O_p(n^{-1/3})$, $T_{i_{\m}}-t_0=O_p(n^{-1/3})$, and therefore all three expressions on the right-hand side of (\ref{random_versus_parabolic}) are $O_p(n^{-2/3})$.

Hence $\m=O_p\left(n^{-2/3}\right)$ and
$$
\f(\m)=\max_{k\le i_0}\min_{i\ge i_0}\frac{\sum_{j=k}^i\d_j+n\m\,a(1-a)}{i-k+1}
\ge\min_{i\ge i_0}\frac{\sum_{j=k_1}^i\d_j+n\m\,a(1-a)}{i-k_1+1}=a.
$$
By the monotonicity and continuity of the function $\f$ we can now conclude
$$
0\le\hat\m_n\le\m=O_p\left(n^{-2/3}\right).
$$
The case $a<\hat F_n(t_0)$ can be treated in a similar way.
\end{proof}

\begin{remark}
\label{remark:kimpol_arguments}
{\rm The crux of the matter in proving a result like $T_{i_{\m}}-t_0=O_p(n^{-1/3})$ in the proof of Lemma 
\ref{lemma:order_mu_CS} (see the discussion below (\ref{random_versus_parabolic})), is that, outside a neighborhood of order $n^{-1/3}$, the last two terms of the three terms on the right-hand side of (\ref{random_versus_parabolic}) cannot cope with the negative parabolic drift of the first term. Arguments of this type are familiar by now, and were for example also used in the proofs of Lemma 3.5 in \cite{piet_geurt:14} and Lemma 5.4  on p. 95 of \cite{GrWe:92}. Arguments of this type can also be found in \cite{kimpol:90}.
}
\end{remark}

The preceding lemmas enable us to prove the following result, which corresponds to Theorem 2.5 in \cite{mouli_jon:01}.
The proof is given in Section \ref{section:appendix1}.

\begin{theorem}
\label{th:LR_current_status}
Let $F_0$ and $G$ be distribution functions with continuous densities $f_0$ and $g$ in a neighborhood of the point $t_0$ such that $0<F_0(t_0)<1$ and $f_0(t_0)$ and $g(t_0)$ are strictly positive. Let $\hat F_n$ be the unrestricted MLE and let $\hat F_n^{(0)}$ be the MLE under the restriction that $\hat F_n^{(0)}(t_0)=F_0(t_0)$. Moreover, let the log likelihood ratio statistic $2\log{\ell_n}$ be defined by
$$
2\log{\ell_n}=2\sum_{i=1}^n\left\{\dd_i\log\frac{\hat F_n(T_i)}{\hat F_n^{(0)}(T_i)}+(1-\dd_i)\log\frac{1-\hat F_n(T_i)}{1-\hat F_n^{(0)}(T_i)}\right\}\,.
$$
Then
$$
2\log{\ell_n}\stackrel{{\cal D}}\longrightarrow \mathbb D,
$$
where $\mathbb D$ is the universal limit distribution as given in \cite{mouli_jon:01}.
\end{theorem}

\subsection*{Construction of SMLE based confidence intervals for the distribution function}
Let $F_0$ be defined on an interval $[a,b]$ with $a<b$ satisfying $F_0(a)=0$ and $F_0(b)=1$. Then we can estimate $F_0$ by the SMLE, using a boundary correction:
\begin{equation}
\label{SMLE_with_bd_correction}
\tilde F_{nh}(t)=\int\left\{\IK\left(\frac{t-x}{h}\right)+\IK\left(\frac{t+x-2a}{h}\right)
-\IK\left(\frac{2b-t-x}{h}\right)\right\}\,d\hat F_n(x),
\end{equation}
where $\hat F_n$ is the MLE, $\IK(x)=\int_{-\infty}^x K(u)\,du$, and $K$ is a symmetric kernel density, like the triweight kernel. If $t\in[a+h, b-h]$ the SMLE coincides with the familiar
$$
\tilde F_{nh}(t)=\int\IK\left(\frac{t-x}{h}\right)\,d\hat F_n(x),
$$
the other two terms in (\ref{SMLE_with_bd_correction}) are only there for correction at the left and right boundary. For simplicity we take $a=0$ in the following (the usual case), and the interval, containing the support of $F_0$ will now be denoted by $[0,b]$.

For the construction of the $1-\a$ confidence interval we take a number of bootstrap samples $(T_1^*,\dd_1^*),\dots,(T_n^*,\dd_n^*)$ with replacement from $(T_1,\dd_1),\dots,(T_n,\dd_n)$. For each such sample we compute the SMLE $\tilde F_{nh}^*$, using the same bandwidth $h$ as used for the SMLE $\tilde F_{nh}$ in the original sample, and the same type of boundary correction. Next we compute at the points $t$:
\begin{align}
\label{Z_n^*}
&Z_{n,h}^*(t)\nonumber\\
&=\frac{\tilde F_{nh}^*(t)-\tilde F_{nh}(t)}{\sqrt{n^{-2}\sum_{i=1}^n\left\{K_h(t-T_i^*)-K_h(t+T_i^*)-K_h(2b-t-T_i^*)\right\}^2\left(\dd_i-\hat F_n^*(T_i^*)\right)^2}}\,,
\end{align}
where $\hat F_n^*$ is the ordinary MLE (not the SMLE!) of the bootstrap sample $(T_1^*,\dd_1^*),\dots,(T_n^*,\dd_n^*)$.

Let $U_{\a}^*(t)$ be the $\a$-th percentile of the $B$ bootstrap values $Z_{n,h}^*(t)$. Then, disregarding the bias for the moment, the following bootstrap $1-\a$ interval is suggested:
\begin{equation}
\label{CI_type1}
\left[\tilde F_{nh}(t)-U_{1-\a/2}^*(t)S_{nh}(t),
\tilde F_{nh}(t)-U_{\a/2}^*(t)S_{nh}(t)\right],
\end{equation}
where
$$
S_{nh}(t)^2=n^{-2}\sum_{i=1}^n\left\{K_h(t-T_i)-K_h(t+T_i)-K_h(2b-t-T_i)\right\}^2\left(\dd_i-\hat F_n(T_i)\right)^2\,.
$$
The bootstrap confidence interval is inspired by the fact that the SMLE is asymptotically equivalent to the toy estimator
\begin{align*}
F_{nh}^{toy}(t)&=\int\left\{\IK_h(t-u)+\IK_h(t+u)-\IK_h(2b-t-u)\right\}\,dF_0(u)\nonumber\\
&\qquad+\frac1n\sum_{i=1}^n\frac{\left\{K_h(t-T_i)-K_h(t+T_i)-K_h(2b-t-T_i)\right\}\,\left\{\dd_i-F_0(T_i)\right\}}{g(T_i)}\,,
\end{align*}
the variance of which can be estimated by
\begin{align*}
S_n(t)^2=\frac1{n^2}\sum_{i=1}^n\frac{\left\{K_h(t-T_i)-K_h(t+T_i)-K_h(2b-t-T_i)\right\}^2\,\left\{\dd_i-F_0(T_i)\right\}^2}{g(T_i)^2},
\end{align*}
and also by Theorem 4.2, p.\  365 in \cite{piet_geurt_birgit:10}, which states that,
if $h\sim cn^{-1/5}$, under the conditions of that theorem, for each $t\in(0,b)$,
$$
n^{2/5}\left\{\tilde F_{nh}(t)-F_0(t)\right\}\stackrel{{\cal D}}\longrightarrow N\left(\mu,\s^2\right),
\qquad n\to\infty,
$$
where
$$
\mu=\tfrac12c^2f_0'(t)\int u^2K(u)\,du
$$
and
$$
\s^2=\frac{F_0(t)\{1-F_0(t)\}}{cg(t)}\int K(u)^2\,du.
$$

We now first study the behavior of intervals of type (\ref{CI_type1}) for a situation where the asymptotic bias plays no role (the uniform distribution) and compare the behavior of the intervals with the confidence intervals, based on LR tests for the MLE.

\subsection*{Simulation for uniform distributions}
We generated $1000$ samples $(T_1,\dd_1),\dots,(T_n,\dd_n)$ by generating $T_1,\dots,T_n$, $n=1000$, from the uniform distribution on $[0,2]$ and, independently, a sample $X_1,\dots,X_n$, also from the uniform distribution on $[0,2]$. If $X_i\le T_i$ we get a value $\dd_i=1$, otherwise $\dd_i=0$. For each such sample $(T_1,\dd_1),\dots,(T_n,\dd_n)$ we generated $1000$ bootstrap samples, and computed the $25$th and $975$th percentile of the values (\ref{Z_n^*}) at the points $t_j=0.02,0.04,\dots,1.98$. On the basis of these percentiles we constructed the confidence intervals (\ref{CI_type1}) for all of the ($99$) $t_j$'s and checked whether $F_0(t_j)$ belonged to it. The percentages of simulation runs that $F_0(t_j)$ did not belong to the interval are shown in Figure \ref{fig:percentages_1000Uniform}. We likewise computed the confidence interval, based on the LR test for the MLE for each $t_j$, and also counted the percentages of times that $F_0(t_j)$ did not belong to the interval. The corresponding confidence intervals for one sample are shown in Figure \ref{fig:CI_onesample_Uniform1000}.

\begin{figure}[!ht]
\begin{subfigure}[b]{0.4\textwidth}
\includegraphics[width=\textwidth]{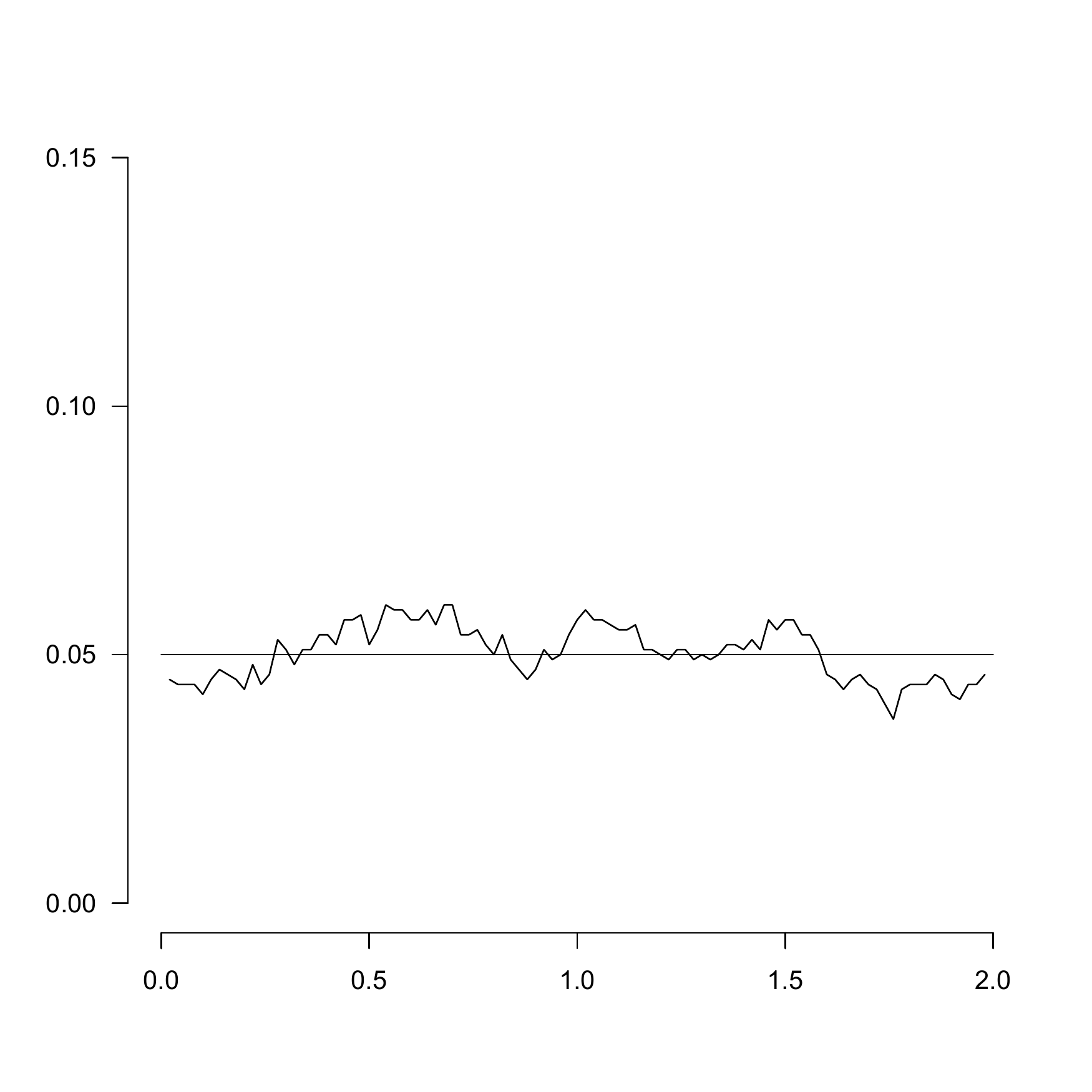}
\caption{}
\label{fig:SMLE_coverage_unif}
\end{subfigure}
\hspace{1cm}
\begin{subfigure}[b]{0.4\textwidth}
\includegraphics[width=\textwidth]{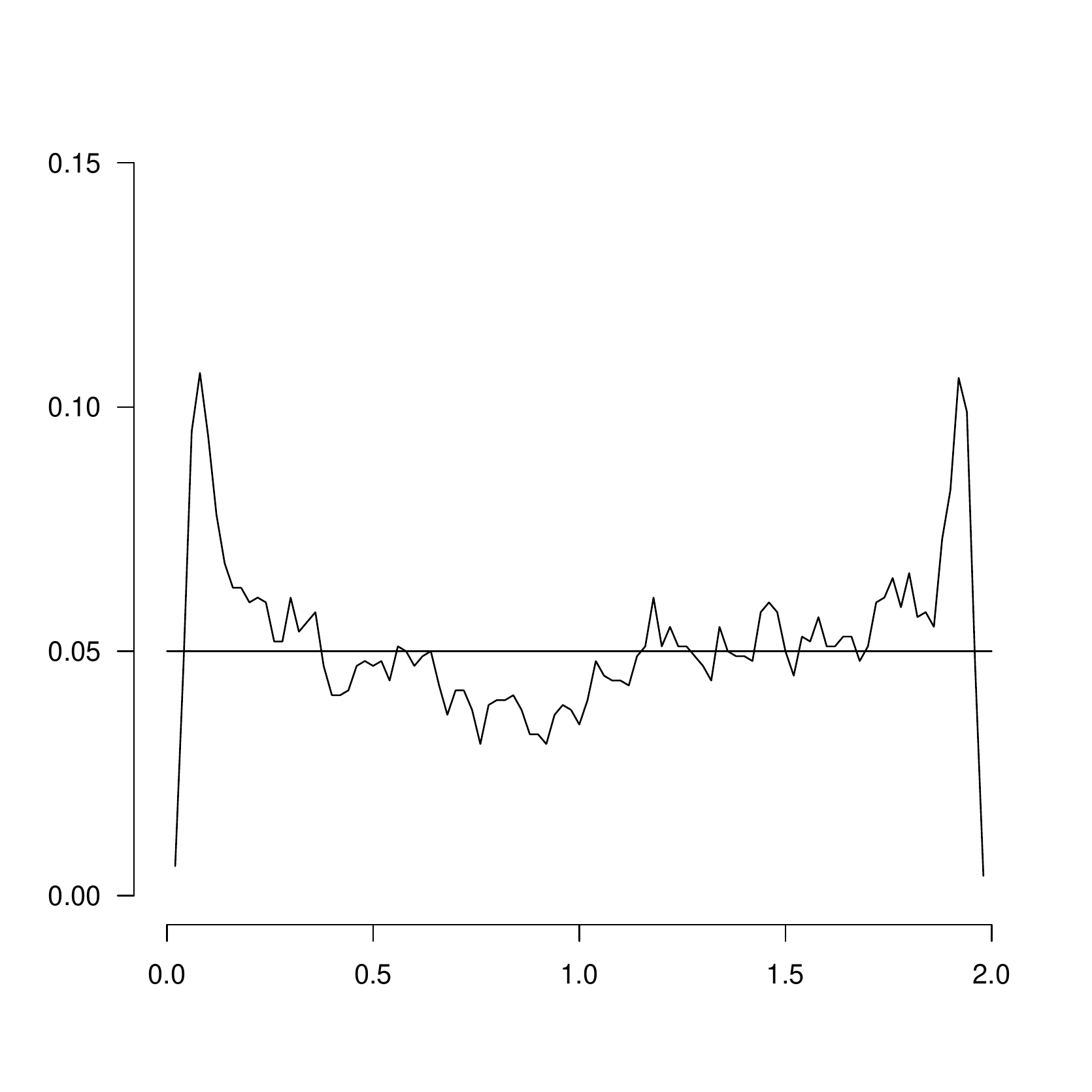}
\caption{}
\label{fig:Jon_Mouli_coverage_unif}
\end{subfigure}
\caption{Uniform samples. Proportion of times that $F_0(t_i),\, t_i=0.02,0.02,\dots,1.98$ is not in the $95\%$ CI's in $1000$ samples $(T_1,\dd_1)\dots,(T_n,\dd_n)$ using the SMLE and $1000$ bootstrap samples from the sample $(T_1,\dd_1)\dots,(T_n,\dd_n)$. In (a), the SMLE is used with CI's given in (\ref{CI_type1}).
In (b) CI's are based on the LR test. The observations are based on two independent samples of $T_i$'s and $_i$'s, uniformly distributed on $[0,2]$.}
\label{fig:percentages_1000Uniform}
\end{figure}

\begin{figure}[!ht]
\begin{subfigure}[b]{0.4\textwidth}
\includegraphics[width=\textwidth]{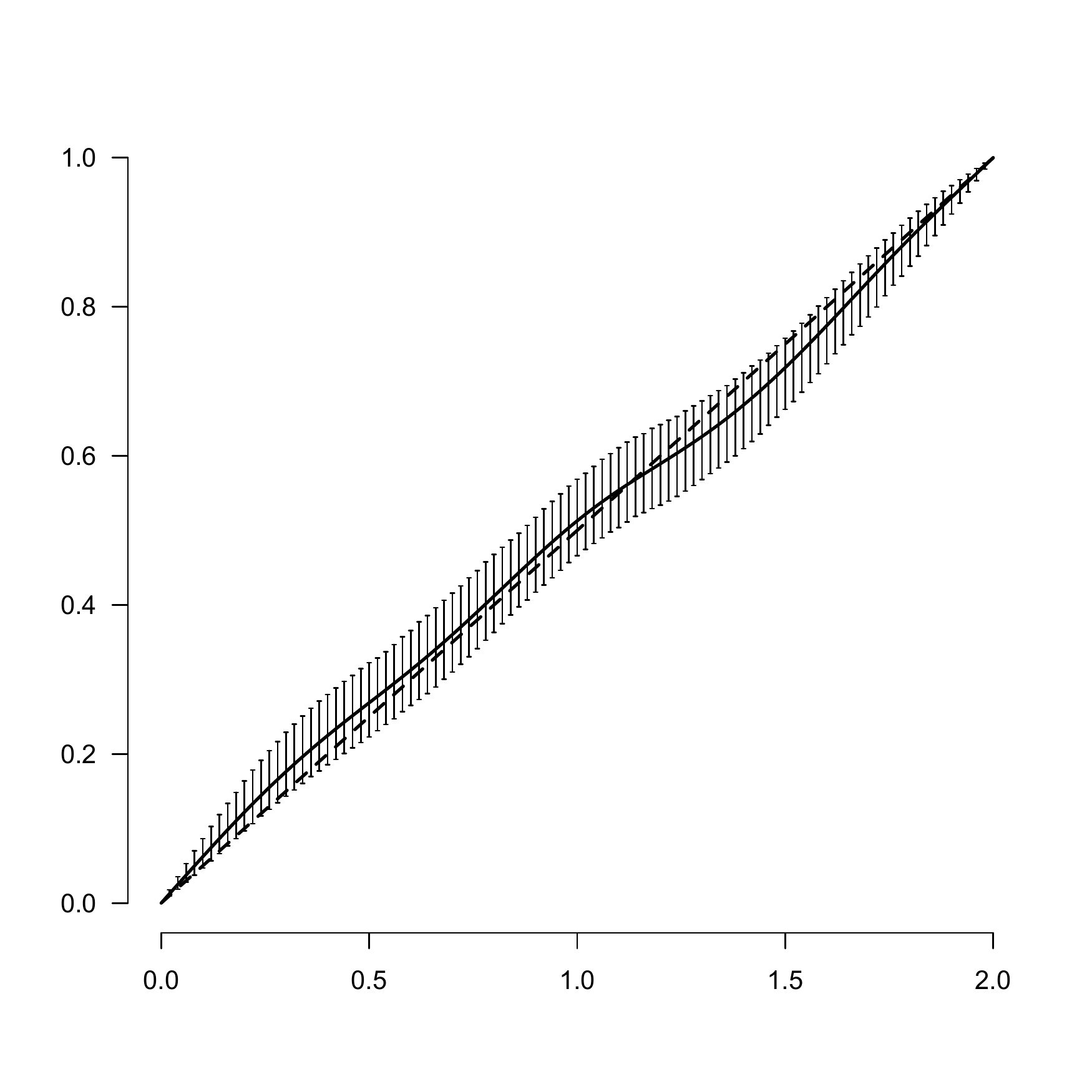}
\caption{}
\label{fig:CI1000_Uniform}
\end{subfigure}
\hspace{1cm}
\begin{subfigure}[b]{0.4\textwidth}
\includegraphics[width=\textwidth]{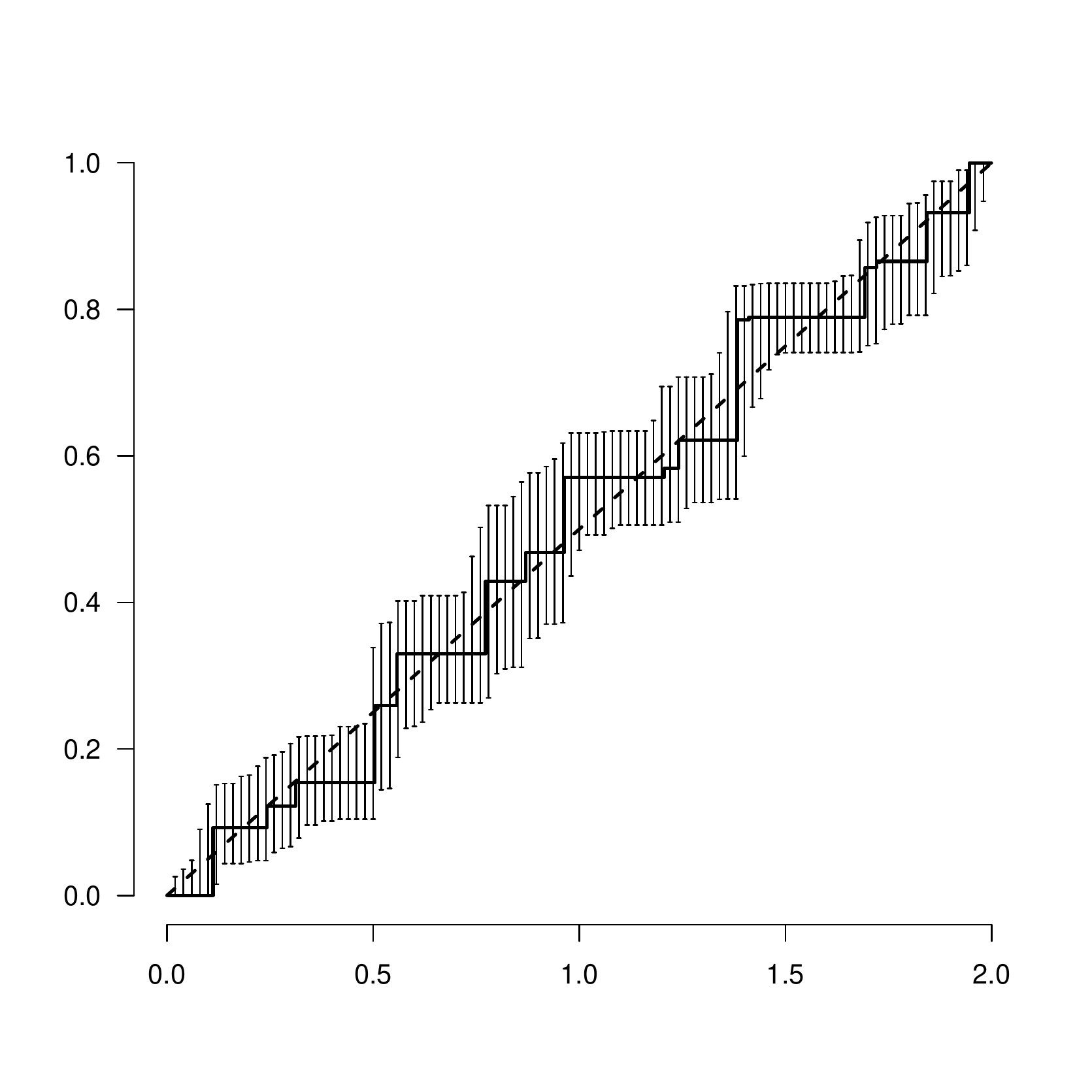}
\caption{}
\label{fig:CI1000_Uniform_Jon_Mouli}
\end{subfigure}
\caption{Uniform samples: $95\%$ confidence intervals for $F_0(t_i),\, t_i=0.02,0.02,\dots,1.98$ for one sample $(T_1,\dd_1)\dots,(T_n,\dd_n)$. For (a) the SMLE and $1000$ bootstrap samples are used; $F_0$ is dashed and the SMLE solid. For (b) the LR test is used; $F_0$ is dashed and the MLE solid.}
\label{fig:CI_onesample_Uniform1000}
\end{figure}

\subsection*{Simulation for truncated exponential distributions}
To investigate the role of the asymptotic bias of the SMLE, we also generated $1000$ samples $(T_1,\dd_1),\dots,(T_n,\dd_n)$ by generating $T_1,\dots,T_n$, $n=1000$, from the uniform distribution on $[0,2]$ and, independently,  $X_1,\dots,X_n$, from the truncated exponential distribution on $[0,2]$, with density
$$
f_0(x)=\frac{e^{-x}}{1-e^{-2}},\,x\in[0,2].
$$
If $X_i\le T_i$ we get $\dd_i=1$, otherwise $\dd_i=0$. For each such sample $(T_1,\dd_1),\dots,(T_n,\dd_n)$ we generated $B=1000$ bootstrap samples, and computed the confidence intervals in the same way as for the uniform samples, discussed above, where the interval is of the form (\ref{CI_type1}) and bias is neglected. This is compared in Figure \ref{fig:percentages_1000TE_biascor} with the results for confidence intervals of the form
\begin{equation}
\label{CI_type2}
\left[\tilde F_{nh}(t)-\b(t)-U_{1-\a/2}^*(t)S_n(t),
\tilde F_{nh}(t)-\b(t)-U_{\a/2}^*(t)S_n(t)\right],
\end{equation}
where $U_{\a/2}^*$, $U_{1-\a/2}^*$ and $S_n(t)$ are as in (\ref{CI_type1}), and where $\b(t)$ is the actual asymptotic bias, which is, for $t\in[h,2-h]$, given by
$$
\tfrac12 f_0'(t)h^2\int u^2 K(u)\,du=-\frac{h^2e^{-t}\int u^2 K(u)\,du}{2\bigl\{1-e^{-2}\bigr\}}\,.
$$
For $t\notin[h,2-h]$ this expression is of the form
$$
-\frac{h^2e^{-t}\left\{\int u^2 K(u)\,du-2\int_{v}^1(u-v)^2K(u)\,du\right\}}{2\bigl\{1-e^{-2}\bigr\}}\,,
$$
where $v=t/h$, if $t\in[0,h)$ and $v=(2-t)/h$ if $t\in(2-h,2]$.

It is seen in Figure \ref{fig:percentages_1000TE_biascor} that if we use the bandwidth $2n^{-1/5}$ and do not use bias correction for the SMLE, the $95\%$ coverage is off at the left end (where the
bias is largest), but that the intervals are `on target' if we add the asymptotic bias to the intervals, as in (\ref{CI_type2}). However, we cannot use the method of Figure \ref{fig:percentage_bias_correction} in practice, since the actual bias will usually not be available.
We are faced here with a familiar problem in nonparametric confidence intervals, and we can take several approaches. Two possible solutions are {\it estimation of the bias} and {\it undersmoothing}.

\begin{figure}[!ht]
\begin{subfigure}[b]{0.4\textwidth}
\includegraphics[width=\textwidth]{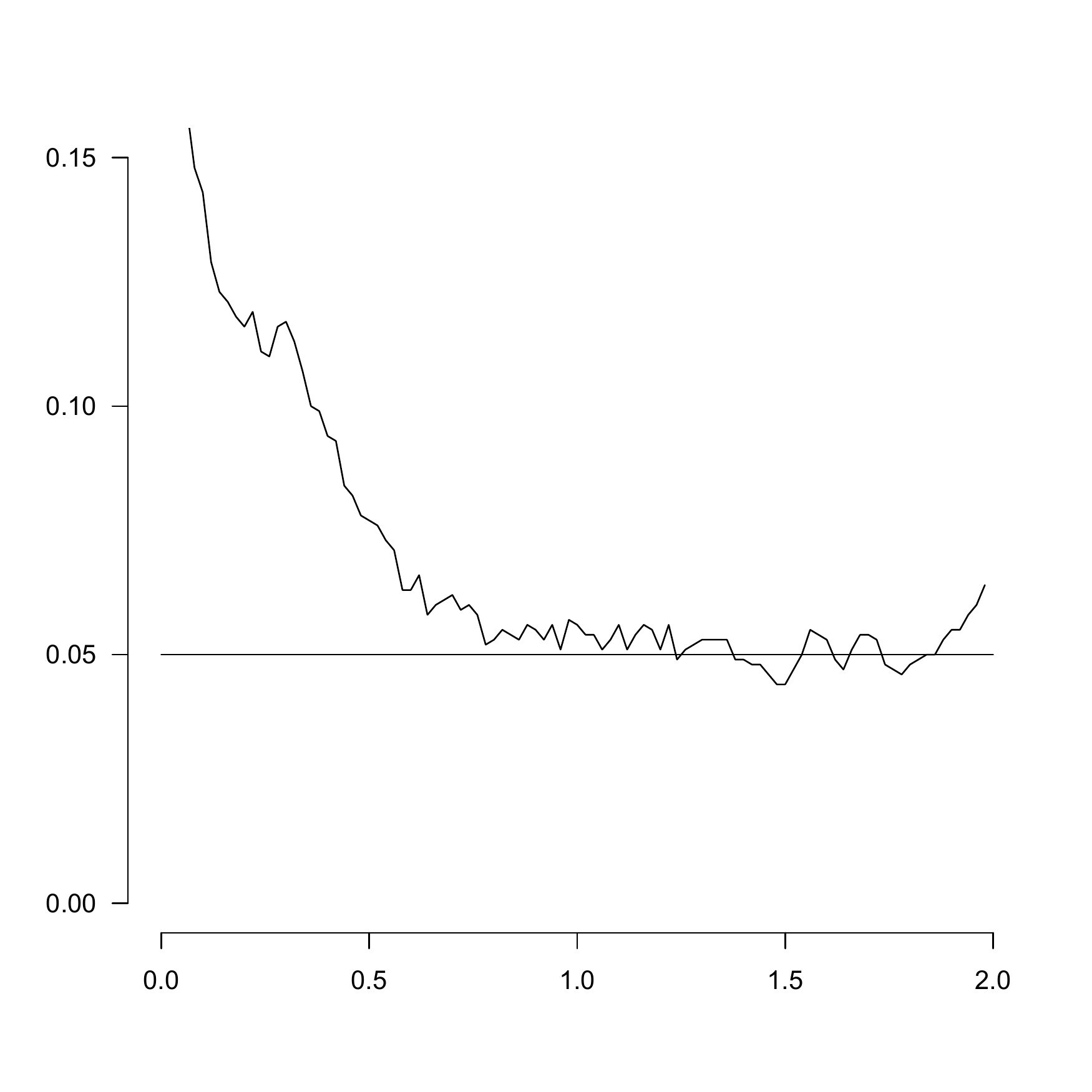}
\caption{}
\label{fig:percentage_uncorrected}
\end{subfigure}
\hspace{1cm}
\begin{subfigure}[b]{0.4\textwidth}
\includegraphics[width=\textwidth]{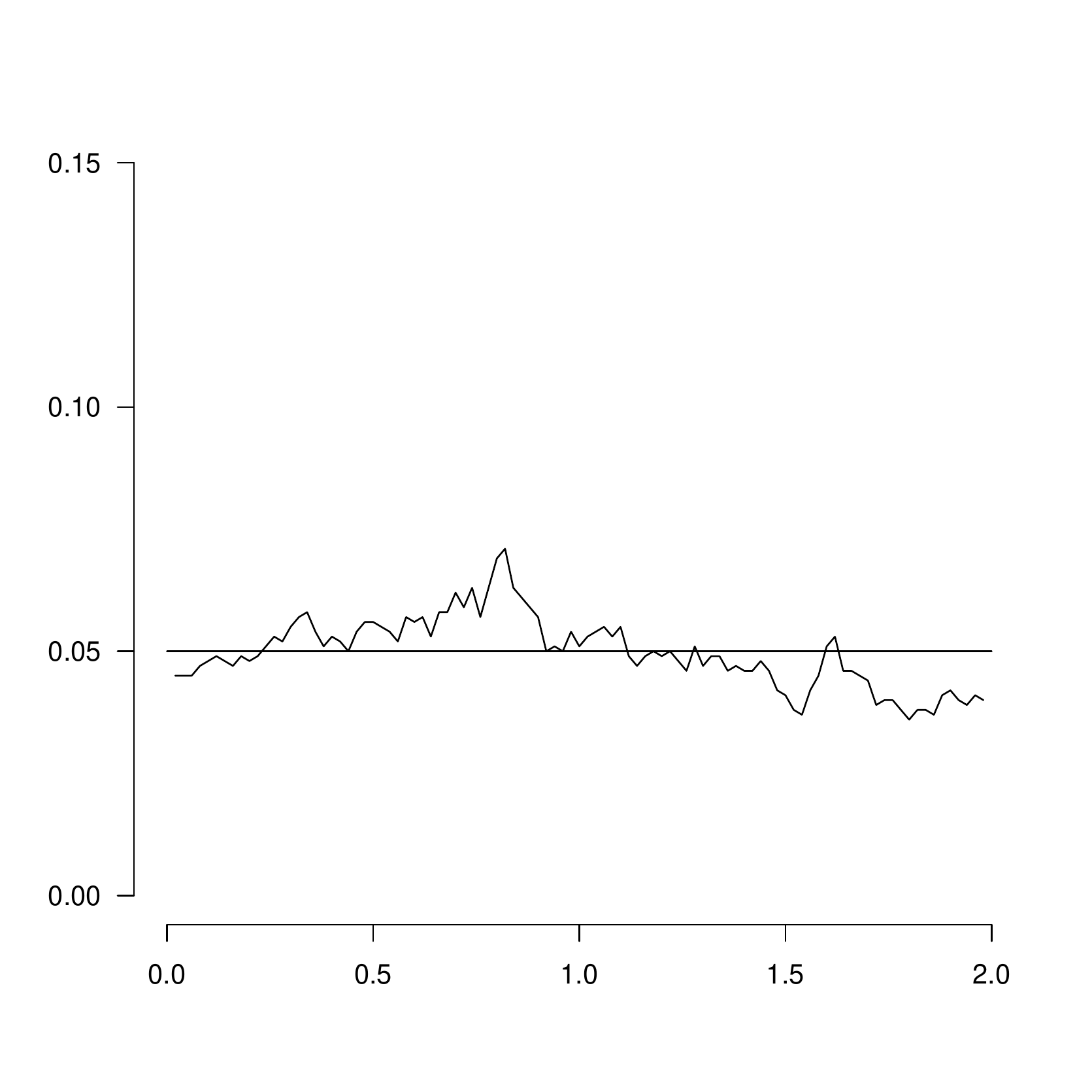}
\caption{}
\label{fig:percentage_bias_correction}
\end{subfigure}
\caption{Coverage for the truncated exponential distribution function $F_0$. Proportion of times that $F_0(t_i),\, t_i=0.01,0.02,\dots$ is not in the $95\%$ CI's in $1000$ samples $(T_1,\dd_1)\dots,(T_n,\dd_n)$. In (a) the confidence intervals (\ref{CI_type1}) are used, in (b) the bias corrected confidence intervals (\ref{CI_type2}). The bandwidth is $h=2n^{-1/5}$}
\label{fig:percentages_1000TE_biascor}
\end{figure}

\begin{figure}[!ht]
\begin{subfigure}[b]{0.4\textwidth}
\includegraphics[width=\textwidth]{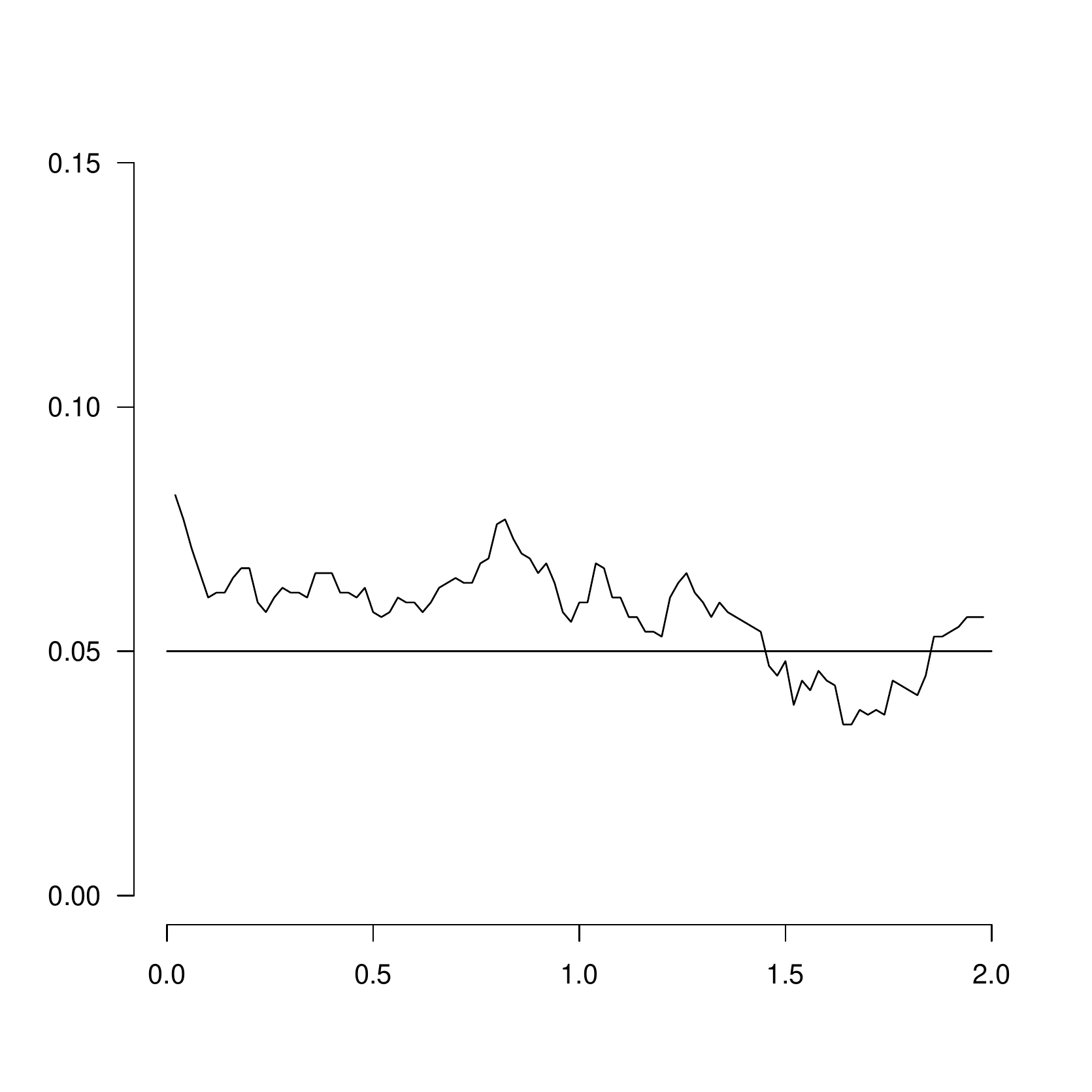}
\caption{}
\label{fig:percentagesTE1000undersmooth}
\end{subfigure}
\hspace{1cm}
\begin{subfigure}[b]{0.4\textwidth}
\includegraphics[width=\textwidth]{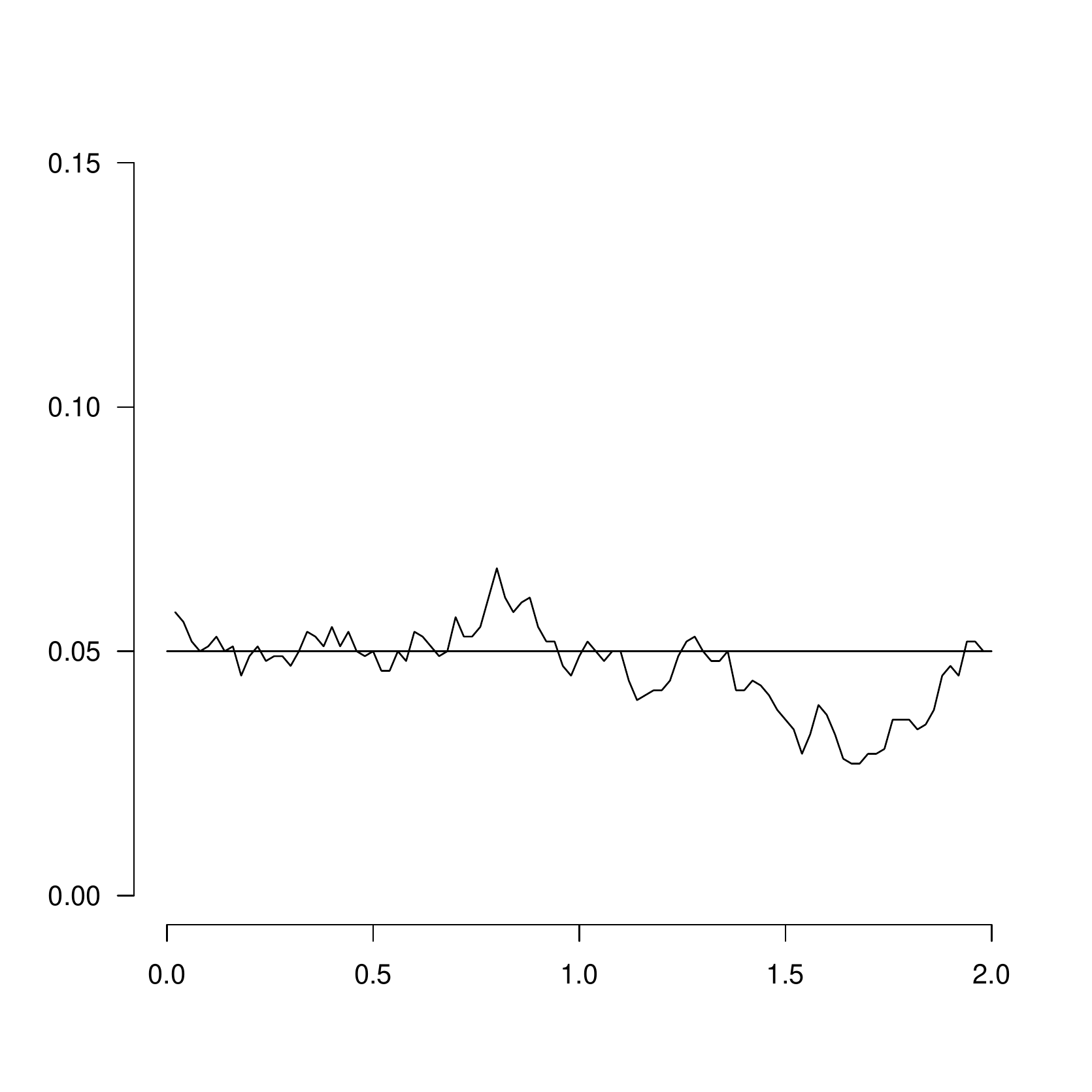}
\caption{}
\label{fig:percentagesTE1000undersmooth2}
\end{subfigure}
\caption{Coverage for the truncated exponential distribution function $F_0$. Proportion of times that $F_0(t_i),\, t_i=0.01,0.02,\dots$ is not in the CI's in $1000$ samples $(T_1,\dd_1)\dots,(T_n,\dd_n)$. In (a) the SMLE and (\ref{CI_type1}) are used for $\a=0.025$ with undersmoothing. In (b), (\ref{CI_type1}) is used with $\a=0.02$ instead of $\a=0.025$ and the same undersmoothing as in (a). The bandwidth is $h=2n^{-1/4}$}
\label{fig:percentages_1000TE}
\end{figure}

In the present case it turns out to be very difficult to estimate the bias term sufficiently accurately. Moreover, \cite{hall:92} argues that undersmoothing has several advantages; one of these is that estimation of the bias term is no longer necessary. For the present model, we changed the bandwidth of the SMLE from $2n^{-1/5}$ to $2n^{-1/4}$ (with $n=1000$) and computed the confidence intervals again by the bootstrap procedure, given above. This gave a remarkable improvement of the coverage at the left end, as is shown in Figure \ref{fig:percentages_1000TE}. Nevertheless, the undersmoothing has the tendency to make the confidence interval slightly liberal (anti-conservative), as can be seen from Figure \ref{fig:percentagesTE1000undersmooth}, so one might prefer to take for example the $20$th and $980$th percentile if one wants to have a coverage $\ge95\%$. The effect of this method is shown in Figure \ref{fig:percentagesTE1000undersmooth2} and the coverage of this method is compared to the coverage of the method, using the LR test, as in \cite{mouli_jon:05}, in Figure \ref{fig:percentages_1000TE2}. Undersmoothing, together with the method of Figure \ref{fig:percentagesTE1000undersmooth2}, will generally of course still produce narrower confidence intervals than the method, based on the LR test (which is based on cube root $n$ asymptotics), under the appropriate smoothness conditions, as can be seen in Figure \ref{fig:CI_onesample_TE1000}.

\begin{figure}[!ht]
\begin{subfigure}[b]{0.4\textwidth}
\includegraphics[width=\textwidth]{percentagesTE1000undersmooth2.pdf}
\caption{}
\label{fig:SMLE_coverage_truncexp}
\end{subfigure}
\hspace{1cm}
\begin{subfigure}[b]{0.4\textwidth}
\includegraphics[width=\textwidth]{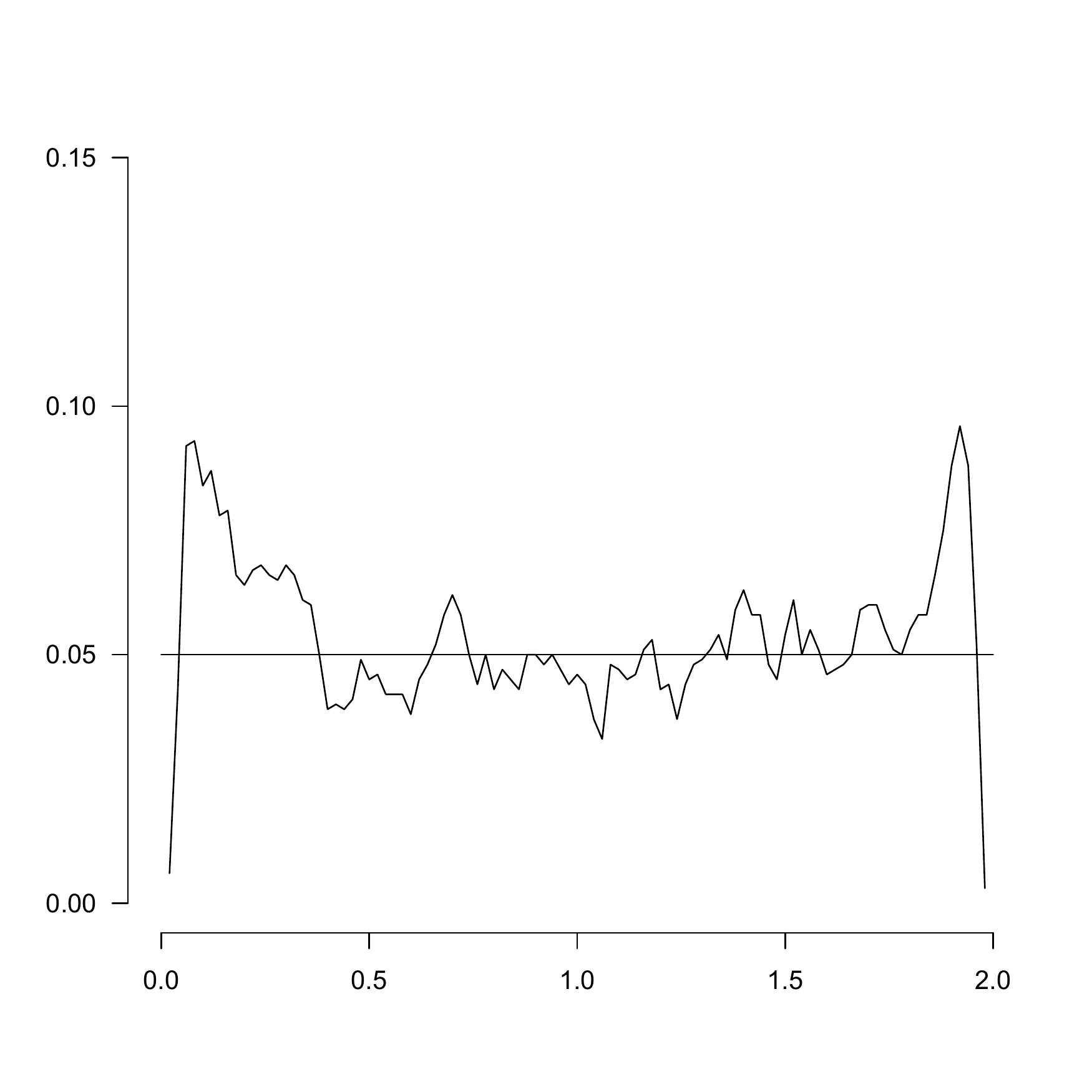}
\caption{}
\label{fig:Jon_Mouli_coverage_TE}
\end{subfigure}
\caption{Truncated exponentials for $F_0$. Proportion of times that $F_0(t_i),\, t_i=0.01,0.02,\dots$ is not in the CI's in $1000$ samples $(T_1,\dd_1)\dots,(T_n,\dd_n)$. Figure (a) uses the SMLE with the method of Figure \ref{fig:percentagesTE1000undersmooth2}. In (b) the LR test for the MLE is used.}
\label{fig:percentages_1000TE2}
\end{figure}

\begin{figure}[!ht]
\begin{subfigure}[b]{0.4\textwidth}
\includegraphics[width=\textwidth]{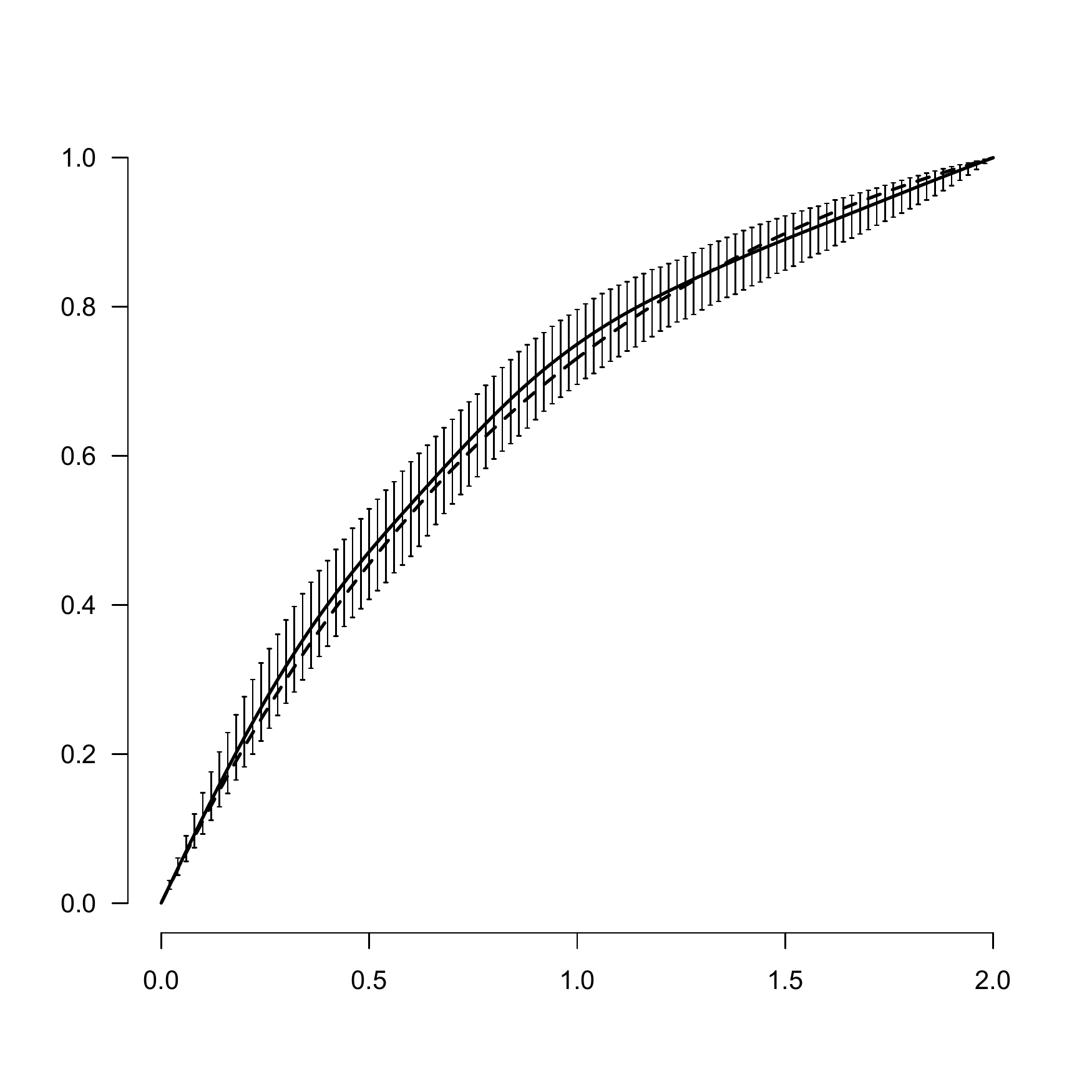}
\caption{}
\label{fig:CI1000_TE}
\end{subfigure}
\hspace{1cm}
\begin{subfigure}[b]{0.4\textwidth}
\includegraphics[width=\textwidth]{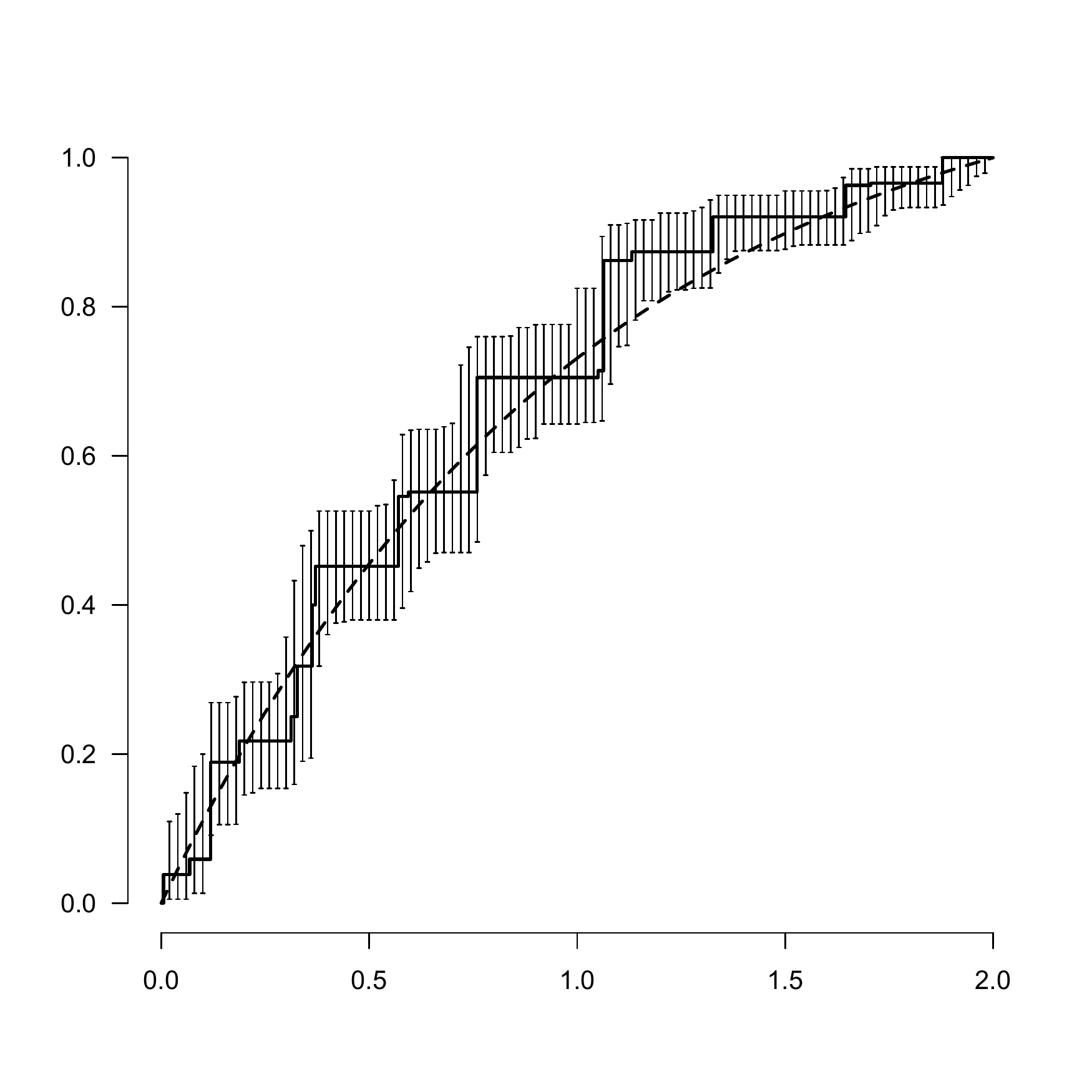}
\caption{}
\label{fig:CI1000_TE_Jon_Mouli}
\end{subfigure}
\caption{Truncated exponentials for $F_0$: $95\%$ confidence intervals for $F_0(t_i),\, t_i=0.01,0.02,\dots$ for one sample $(T_1,\dd_1)\dots,(T_n,\dd_n)$. In (a)  the SMLE is used with undersmoothing and the method of Figure \ref{fig:percentagesTE1000undersmooth2}. Dashed: real $F_0$, solid: SMLE. In (b) the LR test for the MLE is used. Dashed: real $F_0$, solid: MLE.}
\label{fig:CI_onesample_TE1000}
\end{figure}

Another way of bias correction is to use a higher order kernel in the definition of the SMLE, for example a $4$-th order kernel, but still use a bandwidth of order $n^{-1/5}$. Since a $4$-th order kernel has necessarily negative parts, and since the estimate of $F_0$ will be close to zero or $1$ at the boundary of the interval, this gives difficulties at the end of the interval. We therefore stick to the method described above.

\section{Confidence intervals for the monotone density case}
\label{section:monotone_dens}
\setcounter{equation}{0}
In this section we construct confidence intervals for a decreasing density, in the setting of Example \ref{example:monotone_dens}. 
We start by considering the confidence intervals based on the LR tests. To this end, we first give a characterization of the restricted MLE. In view of Example \ref{exam:CurDur} below, in which the observations are on a discrete scale and therefore have ties in the observations, we denote the number of observations at the ordered points $t_i$ by $w_i$. The number of strictly different observation times is denoted by $m$, and the total number of observations is again denoted by $n$, so $n=\sum_{j=1}^m w_j$.

We recall the definition of the unrestricted MLE in this case.

\begin{lemma}
\label{lemma:MLE_dens_unrestricted}
Let $\hat f=(\hat f_1,\dots,\hat f_m)$ be the vector of left-continuous slopes of the least concave majorant of the cusum diagram with points $(0,0)$ and
\begin{equation}
\label{cusum_dens}
\left(t_j,\frac1n\sum_{i=1}^jw_i\right),
\qquad j=1,\dots,m,
\end{equation}
Then $\hat f$ maximizes $\sum_{i=1}^m w_i\log f_i$, under the condition that $f$ is nonincreasing and the side condition $\sum_{i=1}^m f_i\left(t_i-t_{i-1}\right)=1$.
\end{lemma}

For convenience, we provide the proof below.

\begin{proof}
Introducing the Lagrange multiplier $\l$, we get the maximization problem of maximizing
\begin{align}
\label{criterion_LSfunction_dens}
\f_{\l,\m}(f_1,\dots,f_m)&=\frac1n\sum_{i=1}^m w_i\log f_i
-\l\left\{\sum_{i=1}^{m} f_i\left(t_i-t_{i-1}\right)-1\right\},
\end{align}
over the convex cone $C_m=\{(f_1,\dots,f_m):f_1\ge f_2\ge\dots f_m\ge0\}$, where we look for $\hat\l\in\R_+$  such that the maximizer $\hat f=(\hat f_1,\dots,\hat f_n)$ satisfies
$$
\sum_{i=1}^m\hat f_i\left(t_i-t_{i-1}\right)=1.
$$

Using the equality part of the Fenchel conditions for this maximization problem, the solution has to satisfy
\begin{align}
\label{fenchel_eq_monotone}
\left\langle\nabla\f_{\hat\l}(\hat f),\hat f\right\rangle&=\frac1n\sum_{i=1}^m w_i-\hat\l\sum_{i=1}^m\left(t_i-t_{i-1}\right)\hat f_i=1-\hat\l\sum_{i=1}^m\left(t_i-t_{i-1}\right)\hat f_i=1-\hat\l=0.
\end{align}
So $\hat\l=1$.

The generators of the cone $C_m$ are of the form
$$
g_1=(1,0,0,\dots,0,0),\,g_2=(1,1,0,\dots,0,0),\dots, g_m=(1,1,1,\dots,1,1).
$$
The inequality part of the Fenchel conditions can therefore be written as
\begin{align*}
\left\langle\nabla\f_{\hat\l,\hat\m}(\hat f),g_j\right\rangle&=\sum_{i=1}^j \left\{\frac{w_i}{nf_i}-\left(t_i-t_{i-1}\right)\right\}\le0,\qquad j=1,\dots,m.
\end{align*}
Using that $\hat f_m>0$, these conditions are equivalent to:
\begin{align*}
&\sum_{i=1}^j \left\{\frac{w_i}{n}-\left(t_i-t_{i-1}\right)\hat f_i\right\}\le0,\qquad j=1,\dots,m,
\end{align*}
Since, by (\ref{fenchel_eq_monotone}), we also have:
\begin{align*}
&\sum_{i=1}^m \left\{\frac{w_i}{n}-\left(t_i-t_{i-1}\right)\hat f_i\right\}=0,
\end{align*}
this proves our claim.
\end{proof}

We now add the condition $f(t_0)=a$ and proceed in a similar way as in the preceding section to characterize the solution under this restriction. However, because of the side condition that the density integrates to $1$, we cannot allow the density to have a jump in the interval, containing $t_0$, as we did for the current status model in that section, without making further adaptations of the function. In order not to complicate things unnecessarily, we restrict the functions in our set to functions, only having jumps at the observation points, and do not allow jumps at $t_0$. Estimators, arising in this way, will be asymptotically equivalent to the estimators which would allow an extra jump at $t_0$.

\begin{lemma}
\label{lemma:one_lambda2}
Let $t_0\in(t_{i_0-1},t_{i_0})$. 
We define $\hat\m\in\R$ to be the solution (in $\m$) of the equation
\begin{equation}
\label{equation_mu2}
\min_{1\le i\le }\max_{\le j\le m}\frac{\sum_{k=i}^j w_k/n+\m a}{\left(t_j-t_{i-1}\right)}
=a\{1+\m a\}.
\end{equation}
and define $\hat f_i^{(0)}$ by the left-hand slope of the least concave majorant of the cusum diagram with points $(0,0)$ and
cusum diagram with points $(0,0)$ and
\begin{equation}
\label{cusum_mu2}
\left((1+\hat\m a)\,t_j,\sum_{i=1}^j\left\{\frac{w_i}{n}+\hat\m a1_{\{i=i_0\}}\right\}\right),
\qquad j=1,\dots,m,
\end{equation}
Then $\hat f$ maximizes $\sum_{i=1}^m w_i\log f_i$, for non-increasing sequences $(f_1,\dots,f_m)$, under the side conditions $\sum_{i=1}^m f_i\left(t_i-t_{i-1}\right)=1$ and $f(t_{i_0})=a$.

\end{lemma}

\begin{remark}
{\rm
The values of $\hat f_i$ and $\hat f_i^{(0)}$ are defined by left-continuous slopes of a concave majorant, we extend this to piecewise {\it left-continuous functions} $\hat f$ and $\hat f^{(0)}$, having the values $\hat f_i$ and $\hat f_i^{(0)}$ at $t_i$. Note that this differs from the definition of the piecewise {\it right-continuous} distribution functions $\hat F$ and $\hat F^{(0)}$ in the preceding section. Since $\hat f^{(0)}(t_{i_0})=a$ and $t_0\in(t_{i_0-1},t_{i_0})$, we have $\hat f^{(0}(t_0)=a$.
}
\end{remark}

\begin{proof}
Introducing the Lagrange multipliers $\l$ and $\mu$, we get the maximization problem of maximizing
\begin{align}
\label{criterion_LSfunction2}
\f_{\l,\m}(f_1,\dots,f_m)&=\frac1n\sum_{i=1}^m w_i\log f_i
-\l\left\{\sum_{i=1}^{m} f_i\left(t_i-t_{i-1}\right)-1\right\}+\m\left(f_{i_0}-a\right),
\end{align}
over the convex cone $C_m=\{(f_1,\dots,f_m):f_1\ge f_2\ge\dots f_m\ge0\}$, where we look for $(\hat\l,\hat\m)\in\R_+\times\R$  such that the maximizer $\hat f=(\hat f_1,\dots,\hat f_n)$ satisfies
$$
\sum_{i=1}^m\hat f_i\left(t_i-t_{i-1}\right)=1\qquad\text{and}\qquad  \hat f_{i_0}=a.
$$

Using the equality part of the Fenchel conditions for this maximization problem, the solution has to satisfy
\begin{align*}
\left\langle\nabla\f_{\hat\l,\hat\m}(\hat f),\hat f\right\rangle&=\frac1n\sum_{i=1}^m w_i-\hat\l\sum_{i=1}^m\left(t_i-t_{i-1}\right)\hat f_i+\hat\m\hat f_{i_0}\\
&=1-\hat\l\sum_{i=1}^m\left(t_i-t_{i-1}\right)\hat f_i+\hat\m\hat f_{i_0}=1-\hat\l+\hat\m\,a=0.
\end{align*}
This yields the following relation between the two Lagrange multipliers:
\begin{equation}
\label{eq_lambda2}
\hat\m=\frac{\hat\l-1}{a}\,.
\end{equation}

The generators of the cone $C_m$ are of the form
$$
g_1=(1,0,0,\dots,0,0),\,g_2=(1,1,0,\dots,0,0),\dots, g_m=(1,1,1,\dots,1,1).
$$
The inequality part of the Fenchel conditions can therefore be written as
\begin{align*}
\left\langle\nabla\f_{\hat\l,\hat\m}(\hat f),g_j\right\rangle&=\sum_{i=1}^j \left\{\frac{w_i}{nf_i}-\hat\l \left(t_i-t_{i-1}\right)\right\}+\hat\m1_{\{j\ge i_0\}}\\
&=\sum_{i=1}^j \left\{\frac{w_i}{nf_i}-\hat\l \left(t_i-t_{i-1}\right)
+\hat\m1_{\{i=i_0\}}\right\}\le0,\qquad j=1,\dots,m.
\end{align*}
Using that $\hat f_m>0$, these conditions are equivalent to:
\begin{align*}
&\sum_{i=1}^j \left\{\frac{w_i}{n}-\hat\l \left(t_i-t_{i-1}\right)\hat f_i+\hat\m1_{\{i=i_0\}}a\right\}\le0,\qquad j=1,\dots,m,
\end{align*}
which we obtain by multiplying the $i$-th component of the inner product with $\hat f_i$.

We now consider the equation:
\begin{equation}
\label{lambda1_eq}
g(\l,\m,a)=a
\end{equation}
where
$$
g(\l,\m,a)=\min_{1\le i\le i_0}\max_{i_0\le j\le m}\frac{\sum_{k=i}^j w_k/n+\m a}{\l\left(t_j-t_{i-1}\right)}=\frac1{\l}\min_{1\le i\le i_0}\max_{i_0\le j\le m}\frac{\sum_{k=i}^j w_k/n+\m a}{\left(t_j-t_{i-1}\right)}\,.
$$
Note that $g(\l,\m,a)$ is the left hand slope of the least concave majorant of the cusum diagram with points $(0,0)$ and
$$
\left(\l t_j,\sum_{i=1}^j \left\{\frac{w_i}{n}+\m a1_{\{i=i_0\}}\right\}\right),\qquad j=1,\dots,m,
$$
evaluated at $\l t_{i_0}$.
So $g(\l,\m,a)$ should be equal to the value of the restricted MLE at $t_0$ and hence should be equal to $a$.

On the other hand, using the identity $\l=1+a\m$, (\ref{lambda1_eq}) turns into
\begin{align*}
\m=\frac1{a^2}\min_{1\le i\le i_0}\max_{i_0\le j\le m}\frac{\sum_{k=i}^j w_k/n+\m a}{\left(t_j-t_{i-1}\right)}-\frac1a.
\end{align*}
Multiplying by $a^2$ yields (\ref{equation_mu2}).
\end{proof}

The cusum diagram for the restricted MLE is shown in Figure \ref{fig:cusum_monotone+MLE} for a sample of size $n=1000$ from a truncated exponential distribution on $[0,2]$, where we subtract the the line connecting $(0,0)$ and $(\hat{\l}t_m,1+a\hat{\mu})$ for clearer visibility of the difference between the least concave majorant and the values of the cusum diagram. We took $i_0=700$, which gave $t_{i_0}=0.909047$ and a value  $\hat f_n(t_{i_0})=0.519022$ for the unrestricted MLE at $t_{i_0}$. The restricted MLE was specified to have the value $0.519022+0.1=0.619022$ at $t_{i_0}$. The computation of the restricted MLE gave $\hat\m=0.064020$ and $t_{i_0}$ was transformed into the point $0.945073$ on the axis of the cumulative weights by multiplying by $1+a\hat\m$.

The lifting of the cusum diagram at $(1+a\hat\m)t_{i_0}$ is clearly visible in part (a) of Figure \ref{fig:cusum_monotone+MLE}. Part (b) of this figure shows that the unrestricted MLE is {\it globally} changed over the whole interval instead of the only local change of the MLE in the current status model. Nevertheless, the (universal) limit distribution of the log likelihood ratio statistic is the same as in the current status model, as we show below.

\begin{remark}
\label{remark:points_of_jump}
{\rm
Note that it is clear from the geometric construction that the penalty in the cusum diagram will only locally lead to different {\it locations} of points of jump of the restricted MLE on an interval $D_n$ with respect to the unrestricted MLE. Outside $D_n$ the points of jump will be the same. This correspondence also follows from the minmax characterization of the MLEs. The correspondence of the points of jump outside $D_n$ is also clearly visible in part (b) of Figure \ref{fig:cusum_monotone+MLE}, where the restricted and unrestricted MLE are plotted in the same scale.
}
\end{remark}

\begin{figure}[!ht]
\begin{subfigure}[b]{0.4\textwidth}
\includegraphics[width=\textwidth]{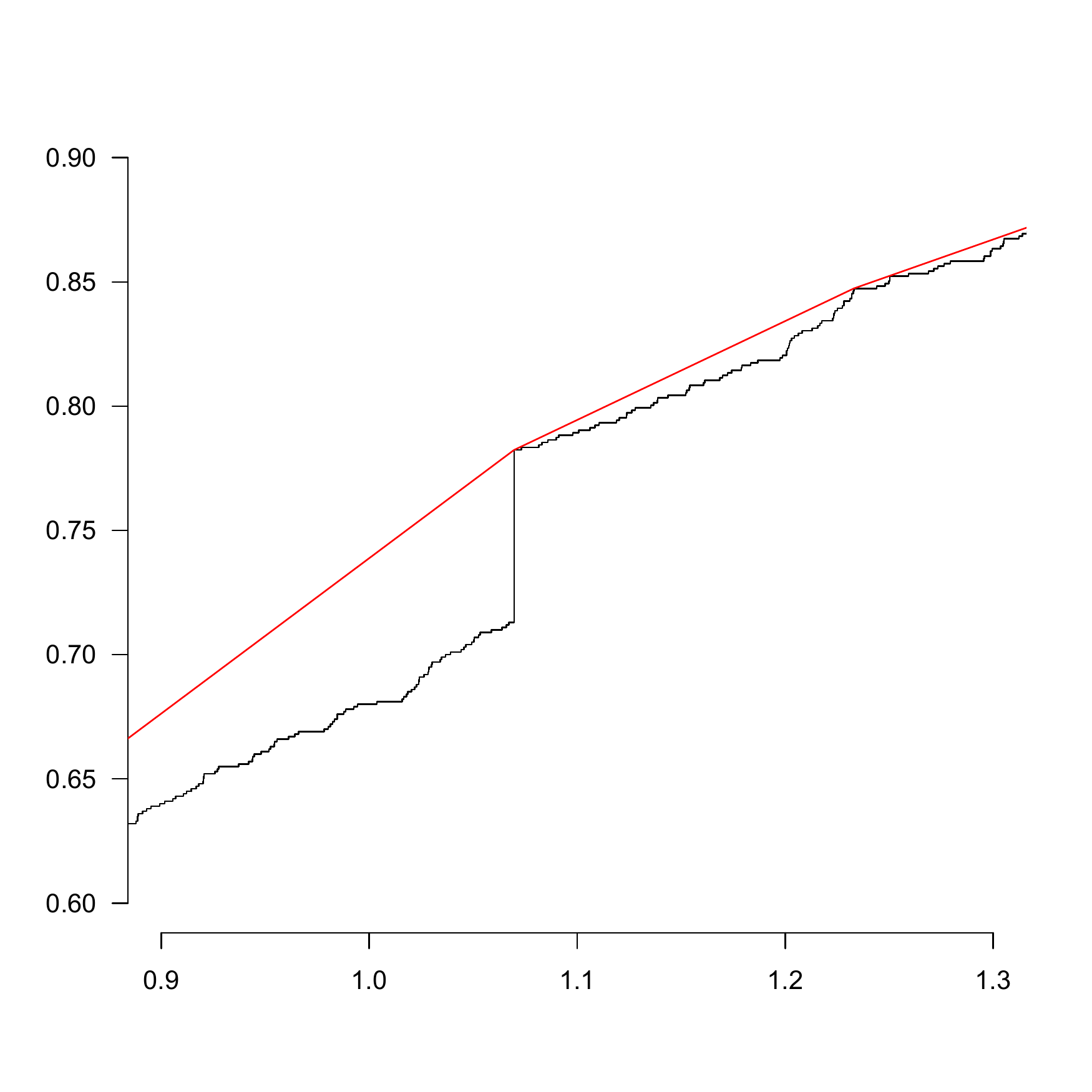}
\caption{}
\label{fig:cusum_monotone}
\end{subfigure}
\hspace{1cm}
\begin{subfigure}[b]{0.4\textwidth}
\includegraphics[width=\textwidth]{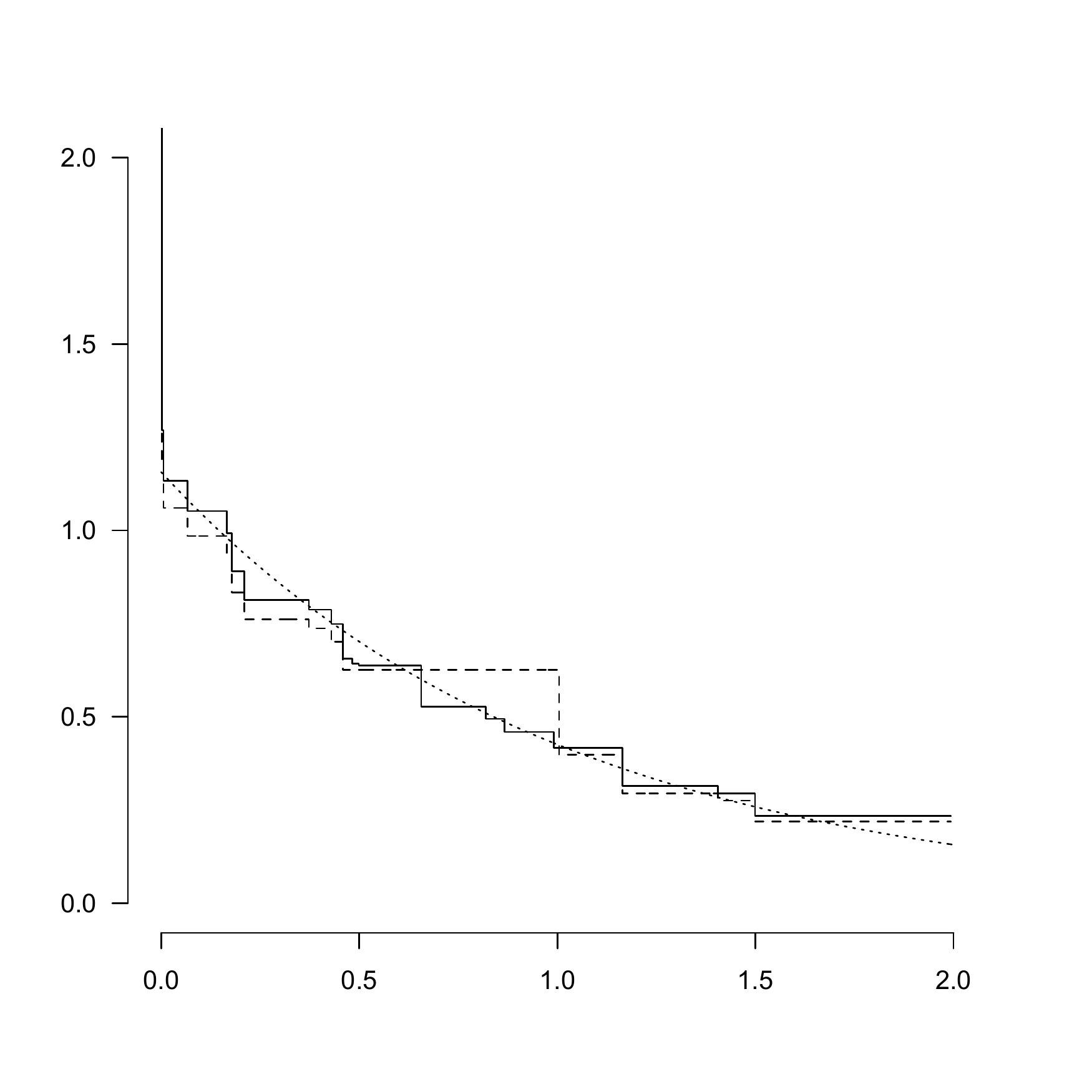}
\caption{}
\label{fig:restricted_monotone}
\end{subfigure}
\caption{Cusum diagram and MLEs for a sample of size $n=1000$ from a truncated exponential distribution with density $f_0$ on $[0,2]$.
We restrict $\hat f^{(0)}$ to have value $a=f_0(1)+0.2$ at $t_0=1$, where $f_0(1)= 0.425459$. (a): cusum diagram with added penalty for the restricted MLE between 0.9 and 1.3. The penalty is added at the location $1.069658=(1+\hat\m a)t_{i_0}$ on the $x$-axis, where $\hat\m=0.10932$ and $t_{i_0}=1.001199$. (b): the restricted MLE (dashed) and the unrestricted MLE (solid).}
\label{fig:cusum_monotone+MLE}
\end{figure}

The proof of Theorem \ref{th:LR_monotone_dens} below will use the following lemma, which is similar to Lemma \ref{lemma:order_mu_CS}.
\begin{lemma}
\label{lemma:order_mu_monotone}
Under the conditions of Theorem \ref{th:LR_monotone_dens} we have, if $a=f_0(t_0)$,
$$
\hat\m_n=O_p\left(n^{-2/3}\right).
$$
\end{lemma}

\begin{proof}
Suppose $t_0\in(t_{i_0-1},t_{i_0})$. Consider the function
$$
\f:\m\mapsto \min_{k\le i_0}\max_{i\ge i_0}\frac{\sum_{j=k}^i w_j/n+\m\,a}{(1+\m a)\bigl(t_i-t_{k-1}\bigr)},\qquad a=f_0(t_0).
$$
By the least concave majorant characterization of the unrestricted MLE $\hat f_n$, we have
$$
\f(0)=\min_{k\le i_0}\max_{i\ge i_0}\frac{\sum_{j=k}^i w_j}{n\bigl(t_i-t_{k-1}\bigr)}=\hat f_n(t_0).
$$

Let $k_1\le i_0$ and $i_1\ge i_0$ be the indices, satisfying
$$
\hat f_n(t_{i_0})=\frac{\sum_{j=k_1}^{i_1} w_j}{n\bigl(t_{i_1}-t_{k_1-1}\bigr)}=\min_{k\le i_0}\max_{i\ge i_0}\frac{\sum_{j=k}^i w_j}{n\bigl(t_i-t_{k-1}\bigr)}\,.
$$
Note that, by the definition of $\hat f_n$, $t_{i_1}$ is the first point of jump (in the sense that $\hat f_n(t)<\hat f_n(t_{i_1})$ if $t>t_{i_1}$) to the right of $t_{i_0}$, and $t_{k_1-1}$ the last point of jump (similarly, $\hat f_n(t)<\hat f_n(t_{k_1-1})$ if $t>t_{k_1-1}$) before $t_{i_0}$ .

Suppose $a>\hat f_n(t_{i_0})$ and let, for $\m>0$, $k_{\m}\le i_0$ be the index such that
$$
\frac{\sum_{j=k_{\m}}^{i_1}w_j/n+\m a}{t_{i_1}-t_{k_{\m}-1}}=\min_{k\le i_0}\frac{\sum_{j=k}^{i_1}w_j/n+\m a}{t_{i_1}-t_{k-1}}
$$
Then, if $a\bigl(t_{i_1}-t_{k_{\m}-1}\bigr)\ne1$, there exists a $\m>0$ such that
$$
\frac{\sum_{j=k_{\m}}^{i_1}w_j+n\m a}{n\bigl(t_{i_1}-t_{k_{\m}-1}\bigr)}=\min_{k\le i_0}\frac{\sum_{j=k}^{i_1}w_j+n\m a}{n\bigl(t_{i_1}-t_{k-1}\bigr)}=a(1+\m a),
$$
and this $\m$ is given by:
$$
\mu=\frac{a\bigl(t_{i_1}-t_{k_{\m}-1}\bigr)-\sum_{j=k}^{i_1} w_j/n}{a\bigl\{1-a\bigl(t_{i_1}-t_{k_{\m}-1}\bigr)\bigr\}}\,.
$$
Using $a=f_0(t_0)$, this can be written in the form:
\begin{equation}
\label{mu_relation2}
0<\m f_0(t_0)=\frac{\int_{t\in(t_{k_{\m}-1},t_{i_1}]}f_0(t_0)\,dt-\int_{t\in(t_{k_{\m}-1},\t_{i_1}]}\,d\F_n(t)}{1-\int_{t\in(t_{k_{\m}-1},\t_{i_1}]}f_0(t_0)\,dt}\,,
\end{equation}
where $\F_n$ is defined by
$$
\F_n(t)=n^{-1}\sum_{i:t_i\le t}w_i.
$$

As noted above, $t_{i_1+1}$ is the first point of jump of $\hat f_n$ to the right of $t_{i_0}$.
Let $\t_+=t_{i_1}$.
As in the proof of Lemma \ref{lemma:order_mu_CS}, we have: $\t_+-t_{i_0}=O_p(n^{-1/3})$. To see this, note that, by (\ref{mu_relation2}), we must have:
$$
\int_{t\in(t_{k_{\m}-1},\t_+]}f_0(t_0)\,dt-\int_{t\in(t_{k_{\m}-1},\t_+]}\,d\F_n(t)>0,
$$
and
\begin{align*}
&\int_{t\in(t_{k_{\m}-1},\t_+]}f_0(t_0)\,dt-\int_{t\in(t_{k_{\m}-1},\t_+]}\,d\F_n(t)\\
&=\int_{t\in(t_{k_{\m}-1},\t_+]}\{f_0(t_0)-f_0(t)\}\,dt-\int_{t\in(t_{k_{\m}-1},\t_+]}\,d\bigl(\F_n-F_0\bigr)(t),
\end{align*}
where the first term on the right gives a negative parabolic drift which cannot be compensated by the second random term outside a neighborhood of order $O_p(n^{-1/3})$ of $t_0$.

By the same type of argument, we can choose for each $\e>0$ an $M>0$ such that
$$
\P\left\{\frac{\int_{u\in(t,\t_+]}f_0(t_{i_0})\,du-\int_{u\in(t,\t_+]}\,d\F_n(u)}{1-
\int_{u\in(t,\t_+]}f_0(t_{i_0})\,du}<0\right\}>1-\e,
$$
if $t<t_{i_0}-Mn^{-1/3}$. But since we must have
$$
\int_{t\in(t_{k_{\m}-1},\t_+]}f_0(t_{i_0})\,dt-\int_{t\in(t_{k_{\m}-1},\t_+]}\,d\F_n(t)>0,
$$
by the positivity of $\m$ and relation (\ref{mu_relation2}), it now follows that $t_{i_0}-t_{k_{\m}-1}=O_p(n^{-1/3})$ and therefore
$$
\m f_0(t_{i_0})=\frac{\int_{t\in(t_{k_{\m}-1},\t_+]}f_0(t_{i_0})\,dt-\int_{t\in(t_{k_{\m}-1},\t_+]}\,d\F_n(t)}{1-
\int_{t\in(t_{k_{\m}-1},\t_+]}f_0(t_{i_0})\,dt}=O_p\left(n^{-2/3}\right).
$$
Hence $\m=O_p\left(n^{-2/3}\right)$
and
$$
\f(\m)=\min_{k\le i_0}\max_{i\ge i_0}\frac{\sum_{j=k}^i w_j/n+a\m}{t_i-t_{k-1}}
\ge\min_{k\le i_0}\frac{\sum_{j=k}^{i_1} w_j/n+a\m}{t_i-t_{k-1}}=a(1+a\m).
$$
As in the proof of Lemma \ref{lemma:order_mu_CS} we can now conclude
$$
0\le\hat\m_n\le\m=O_p\left(n^{-2/3}\right).
$$
The case $a<\hat f_n(t_{i_0})$ can be treated in a similar way.
\end{proof}

We can now prove the following result. The proof is given in Section \ref{section:appendix2}.

\begin{theorem}
\label{th:LR_monotone_dens}
Let $f_0$ be a decreasing density, which is continuous and has a continuous strictly negative derivative $f_0'$ in a neighborhood of $t_0$. Let $\hat f_n$ be the unrestricted MLE and let $\hat f_n^{(0)}$ be the MLE under the restriction that $\hat f_n^{(0)}(t_0)=f_0(t_0)$. Moreover, let the log likelihood ratio statistic $2\log {\ell_n}$ be defined by
$$
2\log{\ell_n}=2\sum_{i=1}^n\log\frac{\hat f_n(T_i)}{\hat f_n^{(0)}(T_i)}\,.
$$
Then
$$
2\log{\ell_n}\stackrel{{\cal D}}\longrightarrow \mathbb D,
$$
where $\mathbb D$ is the universal limit distribution as given in \cite{mouli_jon:01}.
\end{theorem}

\begin{remark}
{\rm The condition that $f_0$ has a continuous strictly negative derivative $f_0'$ in a neighborhood of $t_0$ corresponds to ``condition A" in \cite{mouli_jon:01} for the current status model, which is the condition that the derivative $f_0$ of $F_0$ is strictly positive at $t_0$ and continuous in a neighborhood of $t_0$. A condition of this type is necessary for getting Brownian motion with parabolic drift in the limit distribution of the MLEs. This fails if we take $f_0$ uniform, in which case we get a different type of asymptotics. See section 3.10 in \cite{piet_geurt:14}.
}
\end{remark}

\begin{center}
\begin{figure}
  \includegraphics[width=6.5cm]{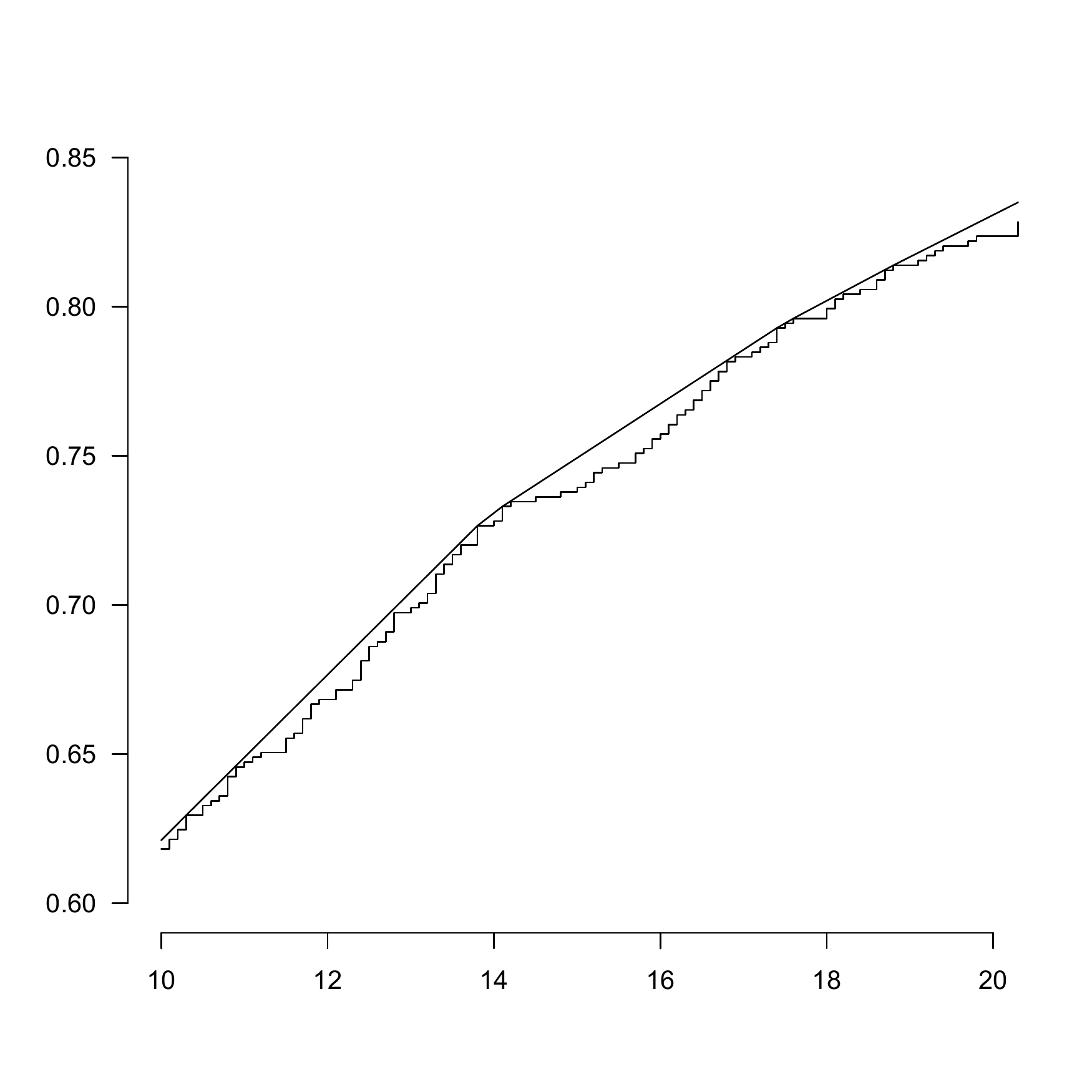}
    \includegraphics[width=6.5cm]{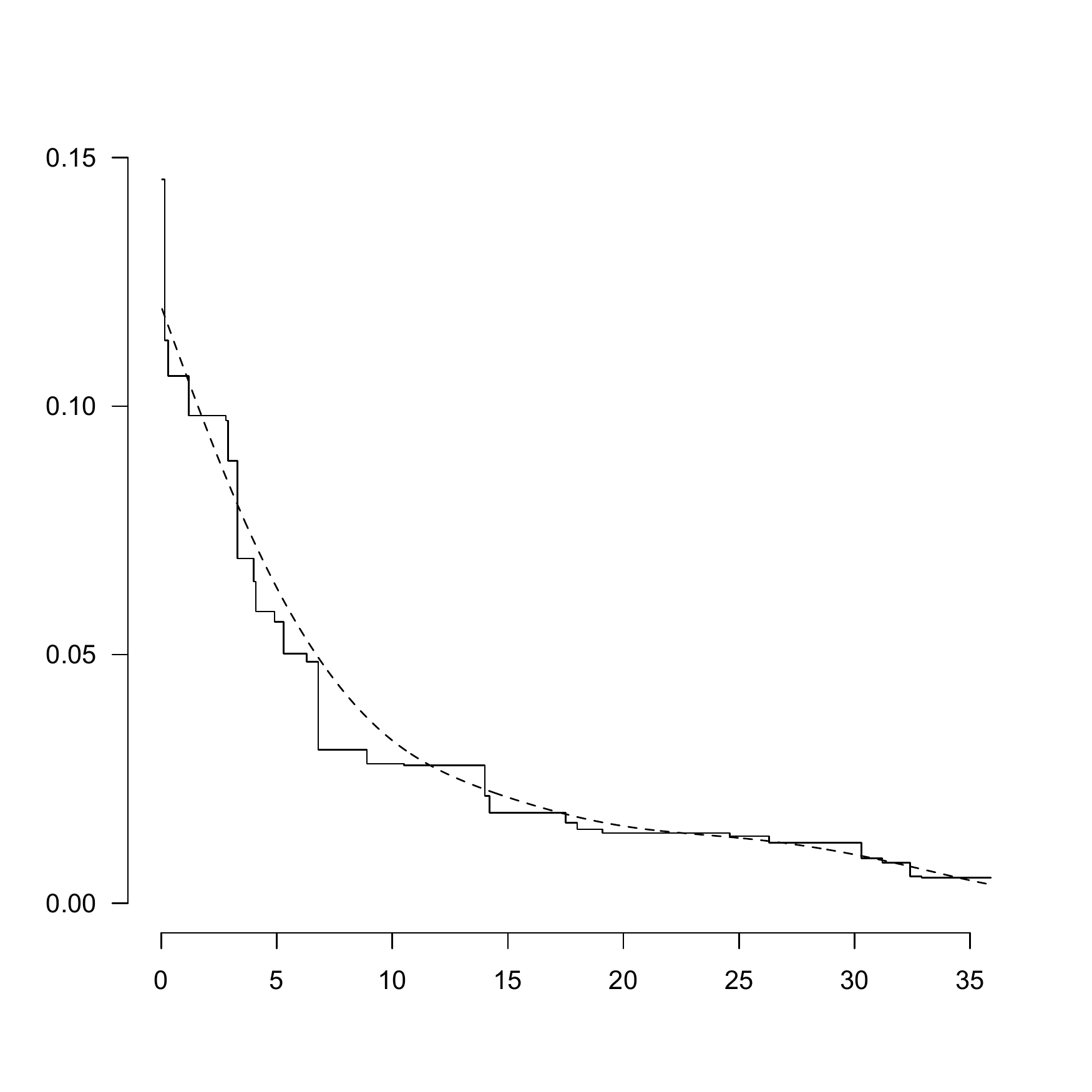}
  \caption{The left panel shows the empirical distribution function and its least concave majorant for the values between $10$ and $20$ months of the 618 current durations $\le$ 36 months. The resulting Grenander estimate (the MLE) of the observation density on the interval $[0,36]$ is shown in the right panel, together with its smoothed version (dashed, the SMLE)}
  \label{fig:curdurecdf}
\end{figure}
\end{center}

\begin{example}
\label{exam:CurDur}
{\rm
Suppose we have a sample $Z_1,\ldots,Z_n$ from the length biased distribution, associated with an unknown distribution function $F$ of interest. This means that the distribution function of $Z_i$ is given by
\begin{equation}
\label{eq:biasedcdf}
\bar{F}(z)=P(Z_i\le z)=\frac1{m_F}\int_0^z x\,dF(x)
\end{equation}
where $m_F=\int_0^{\infty} x\,dF(x)$ is assumed to be nonzero and finite. However, instead of observing the values of $Z_i$ directly, we only observe the data $X_1,\ldots,X_n$ where $X_i$ is a uniform random fraction of $Z_i$. More specifically, we observe
$$
X_i=U_i Z_i,
$$
where $U_1,\ldots,U_n$ is a random sample from the uniform distribution on $[0,1]$, independent of the $Z_i$'s. Now the density of $X_i$ can be seen to be
\begin{equation}
\label{eq:gitoFLP}
g(x)=\frac1{m_F}(1-F(x)),\,\,\,x\ge0,
\end{equation}
see (2.5) in Section 2.2 and Exercise 2.4 in \cite{piet_geurt:14}.
This means that the survival function $1-F(x)$ is given by $g(x)/g(0)$.

Hence, by monotonicity of the initial distribution function $F$ and the fact that $0<m_F<\infty$, it follows that sampling density $g$ is bounded and decreasing on $[0,\infty)$. Moreover, if no additional assumptions are imposed on $F$, any density of this type can be represented by (\ref{eq:gitoFLP}). The density $g$ can be estimated by the Grenander estimator of a decreasing density. See \cite {watson:71} and \cite{vardi:89} for applications of this model.

In  \cite{keidslama:12} a data set of current durations of pregnancy in France is studied. The aim is to estimate the distribution of the time it takes for a woman to become pregnant after having having started unprotected sexual intercourse. For $867$ women the current duration of unprotected intercourse, measured in months, was recorded and this is the basis of part of the research, reported in \cite{keidslama:12}.

Given that the woman in the study is currently trying to become pregnant, the actual recorded data (current duration) can be viewed as uniform random fraction of the true, total duration. In that sense, the model as given in (\ref{eq:gitoFLP}) is not unreasonable. The left panel of Figure \ref{fig:curdurecdf} shows a part of the empirical distribution function of $618$ recorded current durations, kindly provided to us by Niels Keiding, where the data are truncated at $36$ months and are of a similar nature as the data in \cite{keidslama:12}. Based on the least concave majorant, the right panel of Figure \ref{fig:curdurecdf} is computed, showing the resulting MLE of the decreasing density of the observations together with its smoothed version, the smoothed maximum likelihood estimator (SMLE), defined by
\begin{equation}
\label{SMLE_Slama}
\tilde g_{nh}(t)=-\int \IK((t-x)/h)\,d\hat g_n(x),\qquad \IK(x)=\int_x^{\infty} K(u)\,du,
\end{equation}
where $\hat g_n$ is the Grenander estimator (the MLE) and $K$ is a symmetric kernel, for which we took
the triweight kernel
$$
K(u)=\frac{35}{32}\left(1-u^2\right)^31_{[-1,1]}(u),\qquad u\in\R.
$$
The bandwidth $h$ was chosen to be
$$
h=36 n^{-1/5}\approx9.95645,
$$
where $n=618$. Near the boundary points $0$ and $36$ the same boundary correction as in section \ref{section:CI_current_status} was used. For $t\in[h,b-h]$, where $b=36$, the SMLE is asymptotically equivalent to the ordinary kernel density estimator
\begin{equation}
\label{naive_Slama}
\int K_h(t-x)\,d\FF_n(x),\qquad K_h(u)=h^{-1}K(u/h),
\end{equation}
which, however, will in general not be monotone, so not belong to the allowed class.

\begin{figure}[!ht]
\begin{subfigure}[b]{0.4\textwidth}
\includegraphics[width=\textwidth]{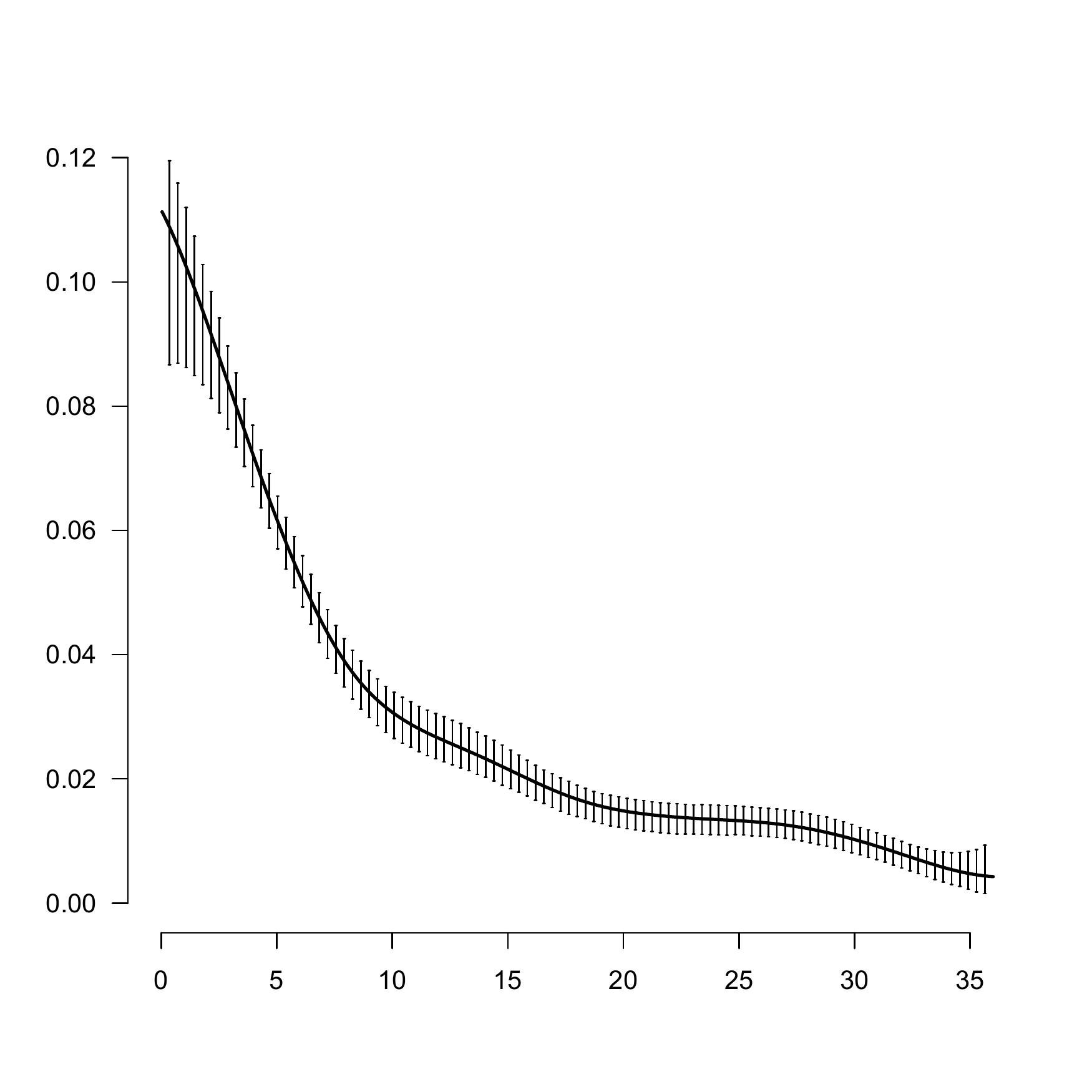}
\caption{}
\label{fig:CI_SMLE_Slama_dens}
\end{subfigure}
\hspace{1cm}
\begin{subfigure}[b]{0.4\textwidth}
\includegraphics[width=\textwidth]{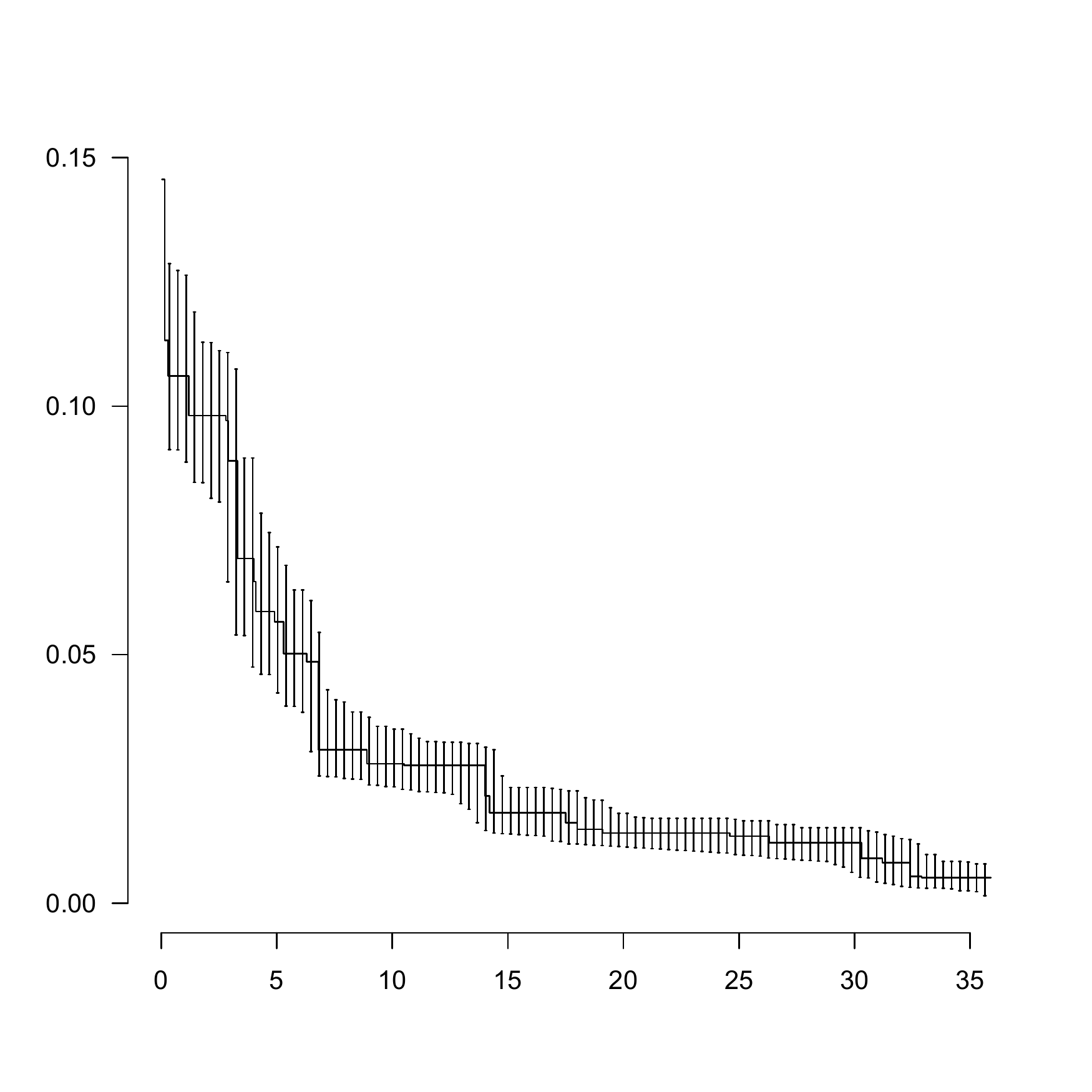}
\caption{}
\label{fig:CI_MLE_Slama_dens}
\end{subfigure}
\caption{$95\%$ confidence intervals, based on the SMLE (part (a)) and MLE (part (b)), respectively, for the data in \cite{keidslama:12} at the points $0.36,0.72,\dots,35.64$. The chosen bandwidth for the SMLE was $36 n^{-1/4}\approx7.2203$.
The time is measured in months.}
\label{fig:Slama_CI_dens}
\end{figure}

\begin{figure}[!ht]
\begin{center}
\includegraphics[scale=0.5]{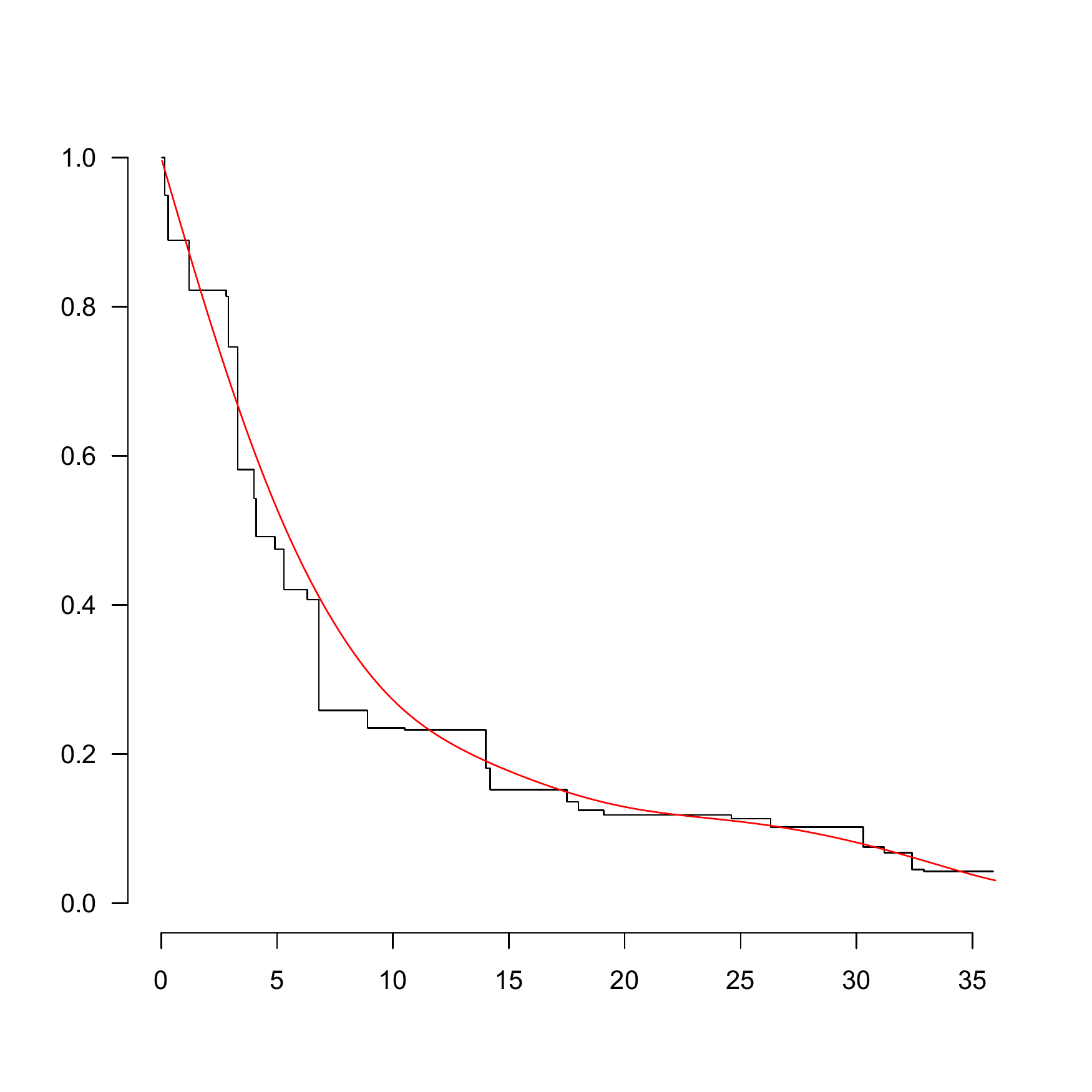}
\end{center}
\caption{Estimates of the survival function, based on the MLE (step function) and SMLE (smooth function), where the MLE is restricted to have the same value at zero as the (consistent) SMLE.}
\label{fig:MLE_restricted_at_zero}
\end{figure}

\begin{figure}[!ht]
\begin{subfigure}[b]{0.4\textwidth}
\includegraphics[width=\textwidth]{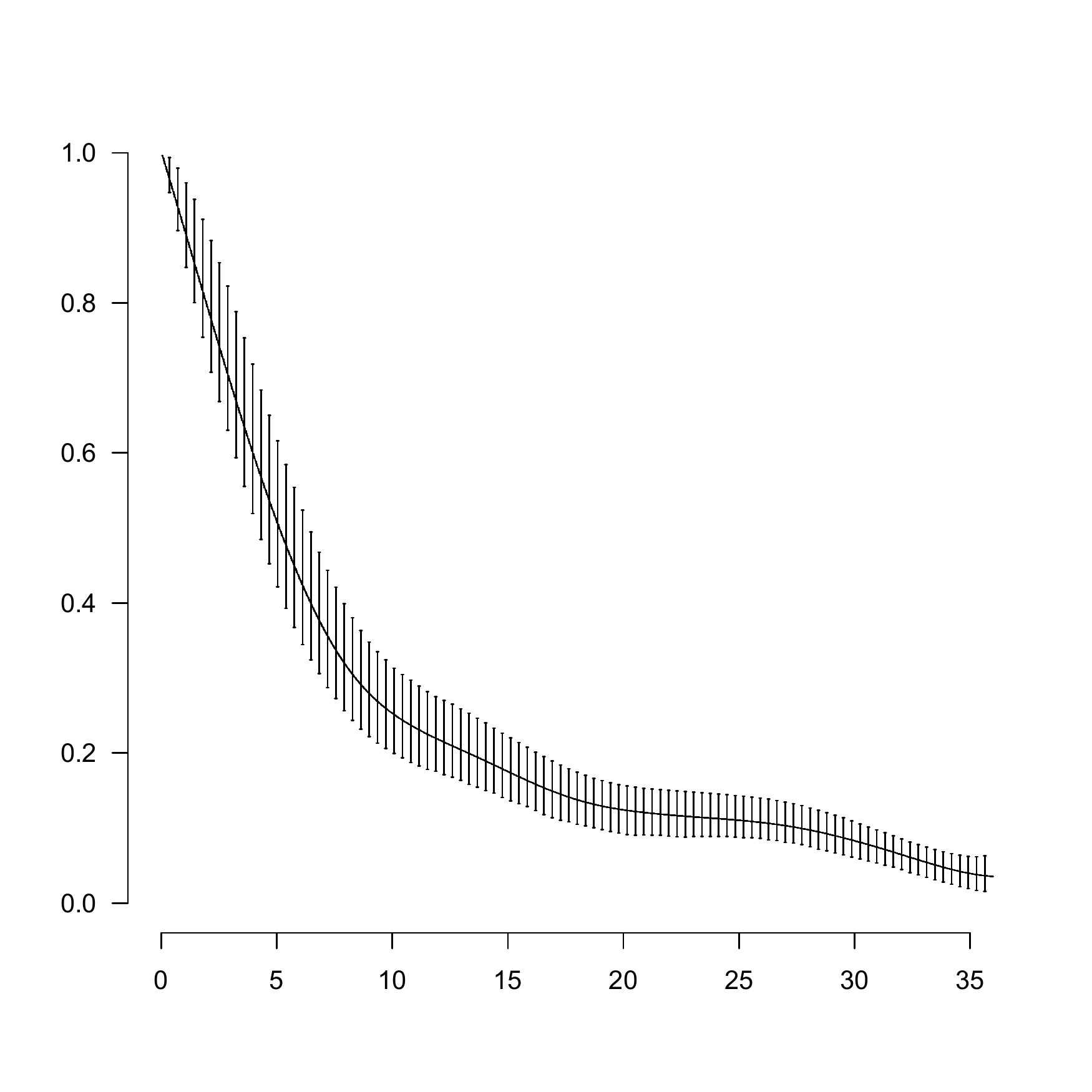}
\caption{}
\label{fig:CI_SMLE_Slama_df}
\end{subfigure}
\hspace{1cm}
\begin{subfigure}[b]{0.4\textwidth}
\includegraphics[width=\textwidth]{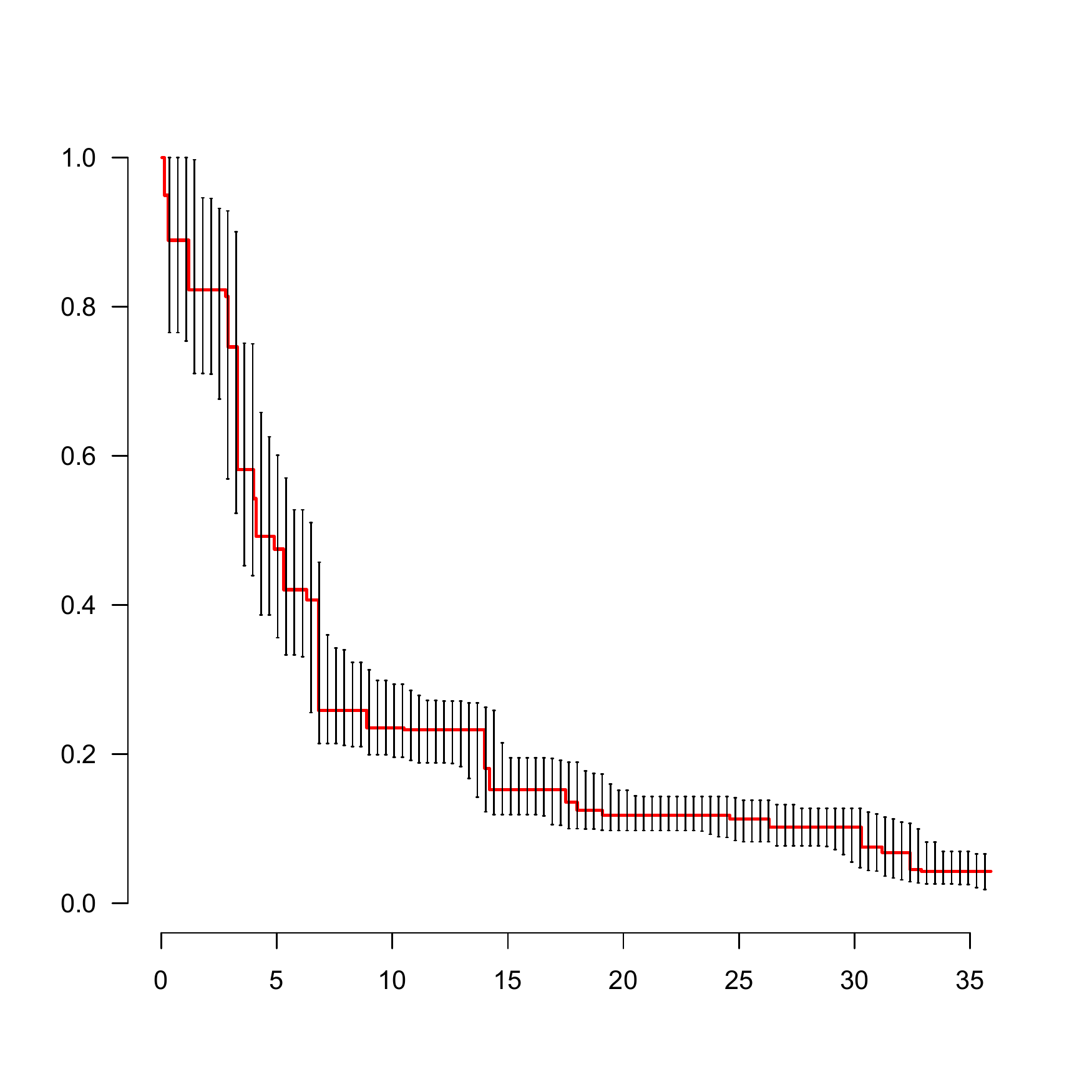}
\caption{}
\label{fig:CI_MLE_Slama_df}
\end{subfigure}
\caption{$95\%$ confidence intervals, based on the SMLE (part (a)) and MLE (part (b)), respectively, for the survival functions in \cite{keidslama:12} at the points $0.36,0.72,\dots,35.64$. The chosen bandwidth for the SMLE was $36 n^{-1/4}\approx7.2203$ and the MLE was restricted to have the same value as the (consistent) SMLE at zero.}
\label{fig:Slama_CI_df}
\end{figure}

The $95\%$ confidence intervals for the density (\ref{eq:gitoFLP}), based on the SMLE and the LR test for the MLE, respectively, are shown in Figure \ref{fig:Slama_CI_dens}.
The survival function for the time until pregnancy or end of the period of unprotected intercourse is given by $g(x)/g(0)$, where $g$ is the density of the observations.
The $95\%$ confidence intervals for the survival function at the $99$ equidistant points $0.36,0.72,\dots,35.64$, are constructed from $1000$ bootstrap samples $T_1^*,\dots,T_n^*$, also of size $n$, drawn from the original sample, and in these samples we computed
\begin{equation}
\label{bootstrap_Slama}
\tilde g_{nh}^*(t)/\tilde g_{nh}^*(0)-\tilde g_{nh}(t)/\tilde g_{nh}(0),
\end{equation}
where $\tilde g_{nh}$ and $\tilde g_{nh}^*$ are the SMLEs in the original sample and the bootstrap sample, respectively. The chosen bandwidth was $36 n^{-1/4}\approx7.2203$, so (according to the method of undersmoothing, see section \ref{section:CI_current_status}), smaller than the bandwidth used in Figure  \ref{fig:curdurecdf}, which uses a bandwidth for which the squared bias and variance are approximately in equilibrium. The $95\%$ asymptotic confidence intervals are given by:
$$
\left[\tilde g_{nh}(t)/\tilde g_{nh}(0)-U^*_{0.975},\tilde g_{nh}(t)/\tilde g_{nh}(0)-U^*_{0.025}\right],
$$
where $U^*_{0.025}$ and $U^*_{0.975}$ are the $2.5\%$ and $97.5\%$ percentiles of the bootstrap values $(\ref{bootstrap_Slama})$. The result is shown in Figure \ref{fig:CI_SMLE_Slama_df} and should be compared with the confidence intervals in part A of Figure 2, p.\ 1495 of \cite{keidslama:12}, based on a parametric (generalized gamma) model.

We have here the easiest, but also somewhat unusual, situation that the isotonic estimator is asymptotically equivalent to an ordinary non-isotonic estimator. The more usual situation is that we only can find a so-called ``toy estimator", which is asymptotically equivalent to the MLE or SMLE, but still contains parameters that have to be estimated. This is the case in the current status model as seen in section \ref{section:CI_current_status}.

In \cite{keidslama:12} and \cite{keiding:12} also parametric models are considered for analyzing these data. We compute the MLE as the slope of the smallest concave majorant of the data $\le 36$ months, where the $x$-values are only the strictly different values, and where we use the number of values at a tie as the increase of the second coordinate of the cusum diagram. In this way we get $618$ values $\le 36$, but only $248$ strictly different ones. It is clear that the SMLE has a somewhat intermediate position w.r.t.\ the parametric models and the fully nonparametric MLE, considered in
\cite{keidslama:12} and \cite{keiding:12}.
}
\end{example}

In the model considered here, the nonparametric MLE is inconsistent at zero and can therefore not be used as an estimate of $g(0)$ and therefore also not as an estimate of the survival function $g(x)/g(0)$, unless we also use penalization at zero. This is in contrast with the SMLE, which is consistent at zero due to the boundary correction. This difficulty with the inconsistency of the MLE at zero for the present model is discussed in \cite{keiding:12}. We solve this difficulty by adding a penalty at zero, as in \cite{WoodSun:93}, and maximize the function
\begin{align}
\label{criterion_LSfunction3}
\f_{\a,\l,\m}(f_1,\dots,f_m)&=\frac1n\sum_{i=1}^m w_i\log f_i
-\l\left\{\sum_{i=1}^{m} f_i\left(t_i-t_{i-1}\right)-1\right\}+\m\left(f_{i_0}-a\right)-\a(f_1-b),
\end{align}
where $b$ is the value of a consistent estimator at zero (for example, the value of the SMLE); we switch to the notation $f=(f_1,\dots,f_m)$ again (instead of using $g$) to be in line with the presentation in the preceding section.
The solution has to satisfy
\begin{align*}
\left\langle\nabla\f_{\hat\a,\hat\l,\hat\m}(\hat f),\hat f\right\rangle&=\frac1n\sum_{i=1}^m w_i-\hat\l\sum_{i=1}^m\left(t_i-t_{i-1}\right)\hat f_i+\hat\m\hat f_{i_0}-\hat\a \hat f_1\\
&=1-\hat\l\sum_{i=1}^m\left(t_i-t_{i-1}\right)\hat f_i+\hat\m\hat f_{i_0}-\a \hat f_1=1-\hat\l+\hat\m\,a-\hat\a b=0,
\end{align*}
and hence
\begin{equation}
\label{eq_lambda3}
\hat\m=\frac{\hat\l-1+\hat\a b}{a}\,.
\end{equation}

Analogously to Lemma \ref{lemma:one_lambda2}, we now get the following lemma.

\begin{lemma}
\label{lemma:two_equations}
Let $\hat f=(\hat f_1,\dots,\hat f_m)$ be the vector of slopes of the least concave majorant of the cusum diagram with points $(0,0)$ and
\begin{equation}
\label{cusum_2restrictions}
\left(\hat\a+\hat\l t_j,\sum_{i=1}^j\left\{\frac{w_i}{n}+(\hat\l-1+\hat\a b){\{i=i_0\}}\right\}\right),
\qquad j=1,\dots,m,
\end{equation}
where $(\hat\a,\hat\l)$ is the solution of the equations (in $(\a,\l)$)
\begin{align}
\label{two_equations}
\min_{1\le i\le i_0}\max_{i_0\le j\le m}\frac{\sum_{k=i}^j w_k/n+\l-1+\a b}{\l\left(t_j-t_{i-1}\right)+1_{\{i=1\}}\a}=a,\qquad\max_{i\ge 1}\frac{\sum_{j=1}^i w_j/n}{\a+\l t_i}=b.
\end{align}
Then $\hat f$ maximizes $\sum_{i=1}^m w_i\log f_i$, under the condition that $f$ is nonincreasing and the boundary conditions
$$
\sum_{i=1}^m f_i\left(t_i-t_{i-1}\right)=1,\qquad f_1=b\qquad\text{ and }f_{i_0}=a.
$$
\end{lemma}

We now restrict the MLE of the density to have a value at zero, given by a consistent estimator at zero. There are several possible choices;  we took the value of the SMLE at zero for illustrative purposes. The resulting estimate of the survival function, based on the MLE restricted at zero to have the same value as the SMLE, is shown in Figure \ref{fig:MLE_restricted_at_zero}. It is also possible to take histogram-type estimates at zero if one wants to impose more lenient conditions. Next we can compute the $95\%$ confidence intervals again by the likelihood ratio method, where one restricts the MLE to have a value at zero, prescribed by the consistent estimate.
Using Lemma \ref{lemma:two_equations} we can then compute the LR tests again for the values of $f_{i_0}$.  The result is shown in part (b) of Figure \ref{fig:Slama_CI_df}, where we used the same asymptotic critical values as before.

\section{Computational aspects and concluding remarks}
\label{section:conclusion}
\setcounter{equation}{0}
There are several ways of computing the restricted MLE's. One way of computing the restricted MLE for the current status model was given in \cite{mouli_jon:01}, see the discussion following Remark \ref{remark:leftcont_curstat} in Section \ref{section:CI_current_status}. We computed the restricted MLE by first solving  equations (\ref{equation_mu}),(\ref{equation_mu2}) or (\ref{two_equations}) for the Lagrange multiplier $\hat\m$ or $\hat\a$ and $\hat\l$, and next computing in one step the left derivative of the greatest convex minorant, resp.\ the smallest concave majorant, of the cusum diagrams which were constructed using the Lagrange multipliers. So the iterative part of the algorithm is in determining the solution $\hat\m$ or $\hat\a$ and $\hat\l$. For the monotone density case it is not clear that a completely non-iterative method for computing the restricted MLE exists (as in the current status  model, if one adapts the definition in terms of inequalities in \cite{mouli_jon:01}). For solving the non-linear equations for $\hat\m$ or $\hat\a$ and $\hat\l$ in Lemma \ref{lemma:two_equations} we wrote C programs, which seems to work fine.

In practice we would recommend to use the methods based on the MLE or SMLE in conjunction; the intervals based on the LR test for the MLE seem pretty much on target, except perhaps for values close to the boundary, and use less assumptions. On the other hand, the intervals, based on the SMLE are narrower and based on asymptotically normal limit distributions, which enables the use of bootstrap methods in constructing the confidence intervals. Direct bootstrap methods have been shown to fail for the MLE, see \cite{kosorok:08} and \cite{sen_mouli_woodroofe:10}.

\section{Appendix A}
\label{section:appendix1}
\setcounter{equation}{0}

\begin{proof}[Proof of Theorem \ref{th:LR_current_status}]
Let $D_n$ be the smallest interval $[a_n,b_n)$ such that $\hat F_n$ and $\hat F_n^{(0)}$ coincide on $D_n^c$ and such that the boundary points of $D_n$ are points of jump of $\hat F_n$ and $\hat F_n^{(0)}$; we assume $\hat F_n$ and $\hat F_n^{(0)}$ to be right-continuous. Then $\hat\m_n=O_p(n^{-2/3})$ and, as argued in the proof of Lemma \ref{lemma:order_mu_CS}, the nearest points of jump to $t_0$ of $\hat F_n^{(0)}$ and $\hat F_n$ are at distance $O_p(n^{-1/3})$ of $t_0$.

Suppose $t_{\ell}>t_{i_1}$, where $(t_{k_1-1},t_{i_1}]$ is the interval around $t_0$ where $\hat F_n^{(0)}$ is constant. The maxmin characterization of $\hat F_n^{(0)}$ then gives
\begin{align*}
\hat F_n^{(0)}(t_{\ell})=\max_{i_0< k\le \ell}\min_{i\ge {\ell}}\frac{\sum_{j=k}^{i}\d_j}{i-k+1}\,.
\end{align*}
Note that the term $\hat\m a1_{\{i=i_0\}}$ does no longer occur in the minmax characterization, since the relevant intervals do not contain $i_0$. Likewise, if $t_{\ell}>t_{i_1'}$, where $(t_{k_1'-1},t_{i_1'}]$ is the interval around $t_0$ where $\hat F_n$ is constant, the maxmin characterization of $\hat F_n$ gives
\begin{align*}
\hat F_n(t_{\ell})=\max_{i_0<k\le \ell}\min_{i\ge{\ell}}\frac{\sum_{j=k}^{i}\d_j}{i-k+1}\,.
\end{align*}

Since we have $t_{i_1}-t_0=O_p(n^{-1/3})$ and $t_{i_1'}-t_0=O_p(n^{-1/3})$, we get therefore that the functions $\hat F_n$ and $\hat F_n^{(0)}$ coincide with high probability for values $t\ge t_0+Mn^{-1/3}$, is $M>0$ is sufficiently large. The same argument holds on intervals to the left of $t_0$. In other words: the length of the interval $D_n=[a_n,b_n)$ is of order $O_p(n^{-1/3})$. By the monotonicity of the functions $\hat F_n$ and $\hat F_n^{(0)}$ and the properties of the unrestricted $\hat F_n$, this also implies:
\begin{equation}
\label{restricted_unrestricted_bounds}
\sup_{t\in D_n}\bigl|\hat F_n(t)-F_0(t_0)\bigr|=O_p\left(n^{-1/3}\right)\qquad\mbox{ and }\qquad\sup_{t\in D_n}\bigl|\hat F_n^{(0)}(t)-F_0(t_0)\bigr|=O_p\left(n^{-1/3}\right).
\end{equation}

We  now have, by (\ref{restricted_unrestricted_bounds}) the Taylor development of the logarithm at the point $F_0(t_0)$, respectively $1-F_0(t_0)$, separately for $\log\hat F_n(t)$, $\log\hat F_n^{(0)}(t)$, etc., and the fact that the length of $D_n$ is of order $O_p(n^{-1/3})$,
\begin{align}
\label{curstat_expansion}
&2n\int_{t\in D_n}\left\{\d\log\frac{\hat F_n(t)}{\hat F_n^{(0)}(t)}+(1-\d)\log\frac{1-\hat F_n(t)}{1-\hat F_n^{(0)}(t)}\right\}\,d\P_n(t,\d)\nonumber\\
&=2n\int_{t\in D_n}\left\{\d\frac{\hat F_n(t)-\hat F_n^{(0)}(t)}{F_0(t_0)}-(1-\d)\frac{\hat F_n(t)-\hat F_n^{(0)}(t)}{1-F_0(t_0)}\right\}\,d\P_n(t,\d)\nonumber\\
&\qquad\qquad-n\int_{t\in D_n}\left\{\d\frac{\bigl\{\hat F_n(t)-F_0(t_0)\bigr\}^2}{F_0(t_0)^2}+(1-\d)\frac{\bigl\{\hat F_n(t)-F_0(t_0)\bigr\}^2}{\bigl\{1-F_0(t_0)\bigr\}^2}\right\}\,d\P_n(t,\d)\nonumber\\
&\qquad\qquad+n\int_{t\in D_n}\left\{\d\frac{\bigl\{\hat F_n^{(0)}(t)-F_0(t_0)\bigr\}^2}{F_0(t_0)^2}+(1-\d)
\frac{\bigl\{\hat F_n^{(0)}(t)-F_0(t_0)\bigr\}^2}{\bigl\{1-F_0(t_0)\bigr\}^2}\right\}\,d\P_n(t,\d)+O_p\left(n^{-1/3}\right)
\end{align}

For the first term on the right-hand side we get:
\begin{align}
\label{treatment_1st_term}
&2n\int_{t\in D_n}\left\{\d\frac{\hat F_n(t)-\hat F_n^{(0)}(t)}{F_0(t_0)}-(1-\d)\frac{\hat F_n(t)-\hat F_n^{(0)}(t)}{1-F_0(t_0)}\right\}\,d\P_n(t,\d)\nonumber\\
&=\frac{2n}{F_0(t_0)\{1-F_0(t_0)\}}\int_{t\in D_n}\{\d-F_0(t_0)\}\{\hat F_n(t)-\hat F_n^{(0)}(t)\}\,d\G_n(t,\d)
\end{align}
We also have:
\begin{align*}
&\int_{t\in D_n}\{\d-F_0(t_0)\}\bigl\{\hat F_n(t)-F_0(t_0)\bigr\}\,d\P_n(t,\d)=\int_{D_n}\{\hat F_n(t)-F_0(t_0)\}^2\,d\G_n(t),
\end{align*}
and
\begin{align*}
&\int_{t\in D_n}\{\d-F_0(t_0)\}\bigl\{\hat F_n^{(0)}(t)-F_0(t_0)\bigr\}\,d\P_n(t,\d)=\int_{D_n}\{\hat F_n^{(0)}(t)-F_0(t_0)\}^2\,d\G_n(t),
\end{align*}
since, by the characterizations of $\hat F_n$ and $\hat F_n^{(0)}$,
$$
\int_{t\in D_n}\{\d-\hat F_n(t)\}\bigl\{\hat F_n(t)-F_0(t_0)\bigr\}\,d\P_n(t,\d)=0,
$$
and
$$
\int_{t\in D_n}\{\d-\hat F_n^{(0)}(t)\}\bigl\{\hat F_n^{(0)}(t)-F_0(t_0)\bigr\}\,d\P_n(t,\d)=0,
$$
where we use that the increments over the $\delta$ coincide with the increments of $\hat F_n$ and $\hat F_n^{(0)}$ between jumps, except for $\hat F_n^{(0)}$ on the interval $[\t_-,\t_+)$ between the successive jumps $\t_-$, $\t_+$, covering $t_0$, where, however $\hat F_n^{(0)}(t)=F_0(t_0)$.
So we obtain:
\begin{align}
\label{result_1st_term}
&\int_{t\in D_n}\{\d-F_0(t_0)\}\bigl\{\hat F_n(t)-F_0(t_0)\bigr\}\,d\P_n(t,\d)\nonumber\\
&=\int_{D_n}\left\{\{\hat F_n(t)-F_0(t_0)\}^2-\{\hat F_n^{(0)}(t)-F_0(t_0)\}^2\right\}\,d\G_n(t).
\end{align}
By (\ref{treatment_1st_term}), this deals with the first term on the right-hand side of (\ref{curstat_expansion}).

To deal with the second and third term of (\ref{curstat_expansion}), we note that
\begin{align*}
&\int_{t\in D_n}\left\{\d\frac{\bigl\{\hat F_n(t)-F_0(t_0)\bigr\}^2}{F_0(t_0)^2}+(1-\d)\frac{\bigl\{\hat F_n(t)-F_0(t_0)\bigr\}^2}{\bigl\{1-F_0(t_0)\bigr\}^2}\right\}\,d\P_n(t,\d)\\
&=\int_{D_n}\left\{\hat F_n(t)\frac{\bigl\{\hat F_n(t)-F_0(t_0)\bigr\}^2}{F_0(t_0)^2}+(1-\hat F_n(t))\frac{\bigl\{\hat F_n(t)-F_0(t_0)\bigr\}^2}{\bigl\{1-F_0(t_0)\bigr\}^2}\right\}\,d\G_n(t)
\end{align*}
again by the fact that since the increments over the $\delta$ coincide with the increments of $\hat F_n$ (note that the integrands on the right-hand side are constant on the intervals of constancy of $\hat F_n$).
This can be written
\begin{align*}
&\int_{D_n}\frac{\hat F_n(t)-F_0(t_0)\bigr\}^2}{F_0(t_0)\{1-F_0(t_0)\}}\,d\G_n(t)
+\int_{D_n}\{\hat F_n(t)-F_0(t_0)\}^3\left\{\frac{1}{F_0(t_0)^2}-\frac1{\{1-F_0(t_0)\}^2}\right\}\,d\G_n(t)\\
&=\int_{D_n}\frac{\hat F_n(t)-F_0(t_0)\bigr\}^2}{F_0(t_0)\{1-F_0(t_0)\}}\,d\G_n(t)+O_p(n^{-4/3}).
\end{align*}
For the same reasons, but using in addition that $\hat F_n^{(0)}(t)=F_0(t_0)$ on the interval of constancy of $\hat F_n^{(0)}$, containing $t_0$, we get:
\begin{align*}
&\int_{t\in D_n}\left\{\d\frac{\bigl\{\hat F_n^{(0)}(t)-F_0(t_0)\bigr\}^2}{F_0(t_0)^2}+(1-\d)\frac{\bigl\{\hat F_n^{(0)}(t)-F_0(t_0)\bigr\}^2}{\bigl\{1-F_0(t_0)\bigr\}^2}\right\}\,d\P_n(t,\d)\\
&=\int_{D_n}\frac{\hat F_n^{(0)}(t)-F_0(t_0)\bigr\}^2}{F_0(t_0)\{1-F_0(t_0)\}}\,d\G_n(t)+O_p(n^{-4/3}).
\end{align*}

Combining the preceding results, we get:
\begin{align}
\label{curstat_expansion2}
&2n\int_{t\in D_n}\left\{\d\log\frac{\hat F_n(t)}{\hat F_n^{(0)}(t)}+(1-\d)\log\frac{1-\hat F_n(t)}{1-\hat F_n^{(0)}(t)}\right\}\,d\P_n(t,\d)\nonumber\\
&=\frac{n}{F_0(t_0)\{1-F_0(t_0)\}}\int_{D_n}\left\{\{\hat F_n(t)-F_0(t_0)\}^2-\{\hat F_n^{(0)}(t)-F_0(t_0)\}^2\right\}\,d\G_n(t)
+O_p\left(n^{-1/3}\right)\nonumber\\
&=\frac{ng(t_0)}{F_0(t_0)\{1-F_0(t_0)\}}\int_{D_n}\left\{\{\hat F_n(t)-F_0(t_0)\}^2-\{\hat F_n^{(0)}(t)-F_0(t_0)\}^2\right\}\,dt+O_p\left(n^{-1/3}\right).
\end{align}

This means that the dominant term of the log likelihood ratio equals
$$
L_n\stackrel{\text{\small def}}=\frac{g(t_0)}{F_0(t_0)\{1-F_0(t_0)\}}\int_{n^{1/3}(a_n-t_0)}^{n^{1/3}(b_n-t_0)}\left\{X_n(t)^2-Y_n(t)^2\right\}\,dt,
$$
where $X_n$ and $Y_n$ are as defined on p.\ 1723 of \cite{mouli_jon:01}:
$$
X_n(t)=n^{1/3}\bigl\{\hat F_n(t_0+n^{-1/3}t)-F_0(t_0)\bigr\},\qquad Y_n(t)=n^{1/3}\bigl\{\hat F_n^{(0)}(t_0+n^{-1/3}t)-F_0(t_0)\bigr\},
$$
see also Theorem 2.4 on p.\ 1710 of \cite{mouli_jon:01}. The resulting convergence of $L_n$ to the universal limit distribution $\mathbb D$ now follows from the joint convergence of $(X_n,Y_n)$ on bounded intervals, as stated in part B of Theorem 2.4 of \cite{mouli_jon:01}, together with Brownian scaling.
\end{proof}

\section{Appendix B}
\label{section:appendix2}
\setcounter{equation}{0}

\begin{proof}[Proof of Theorem \ref{th:LR_monotone_dens}]
We extend the values $\hat f_{ni}$ and $\hat f_{ni}^{(0)}$ of the solution $\hat f_n$ and $\hat f_n^{(0)}$ as vectors to left-continuous functions $\hat f_n$ and $\hat f_n^{(0)}$ on $[0,\infty)$. 
Let $D_n$ be the smallest interval $(a_n,b_n]$ such that $\hat f_n$ and $\hat f_n^{(0)}$ have the same points of jump on $D_n^c$ and such that the boundary points of $D_n$ are points of jump of $\hat f_n$ and $\hat f_n^{(0)}$ (see Remark \ref{remark:points_of_jump}). This means that for $t\notin D_n$:
\begin{align*}
\hat f_n^{(0)}(t)=
\frac1{1+\hat\m_n a}\min_{i: t_i\le t}\,\max_{j:t_j\le t}\frac{\sum_{k=i}^j w_k/n}{t_j-t_{i-1}}
=\frac{\hat f_n(t)}{1+\hat\m_n a},
\end{align*}
since the scale of first coordinates of the cusum diagram for $\hat f_n^{(0)}$ has the factor $1+\hat\m_n a$.
Since $\hat\m_n=O_p(n^{-2/3})$, we get:
\begin{align*}
&2n\int_{D_n^c}\,\log\frac{\hat f_n(t)}{\hat f_n^{(0)}(t)}\,d\F_n(t)=2n\log\{1+a\hat\m_n\}\int_{D_n^c}\,d\F_n(t)
=2n a\hat\m_n\int_{D_n^c}\,d\F_n(t)+O_p\left(n^{-1/3}\right).
\end{align*}
The function $\hat f_n^{(0)}$ must satisfy
$$
\int \hat f_n^{(0)}(x)\,dx=1.
$$
So
$$
\int_{D_n}\hat f_n^{(0)}(t)\,dt+\int_{D_n^c}\hat f_n^{(0)}(t)\,dt=1,
$$
and hence, using $\hat\m_n=O_p(n^{-2/3})$,
\begin{align*}
&2n\int_{D_n}\hat f_n^{(0)}(t)\,dt=2n\left\{1-\int_{D_n^c}\hat f_n^{(0)}(t)\,dt\right\}
=2n\left\{1-\{1+\hat\m_n a\}^{-1}\int_{D_n^c}\,d\F_n(t)\right\}\\
&=2n\left\{\int_{D_n}\,d\F_n(t)+\hat\m_n a\int_{D_n^c}\,d\F_n(t)\right\}+O_p\left(n^{-1/3}\right)\\
&=2n\left\{\int_{D_n}\hat f_n(t)\,dt+\hat\m_n a\int_{D_n^c}\,d\F_n(t)\right\}+O_p\left(n^{-1/3}\right)\\
\end{align*}
So we get:
\begin{align*}
&2n\hat\m_n a\int_{D_n^c}\,d\F_n(t)=2n\int_{D_n}\{\hat f_n^{(0)}(t)-\hat f_n(t)\}\,dt+O_p\left(n^{-1/3}\right).
\end{align*}

So we obtain
\begin{align*}
&2n\int\log \frac{\hat f_n(t)}{\hat f_n^{(0)}(t)}\,d\F_n(t)\\
&=2n\int_{D_n}\log \frac{\hat f_n(t)}{\hat f_n^{(0)}(t)}\,d\F_n(t)
+2n\int_{D_n^c}\log \frac{\hat f_n(t)}{\hat f_n^{(0)}(t)}\,d\F_n(t)\\
&=2n\int_{D_n}\log \frac{\hat f_n(t)}{\hat f_n^{(0)}(t)}\,d\F_n(t)+
2n a\hat\m_n\int_{D_n^c}\,d\F_n(t)+O_p\left(n^{-1/3}\right)\\
&=2n\int_{D_n}\log \frac{\hat f_n(t)}{\hat f_n^{(0)}(t)}\,d\F_n(t)
-2n\int_{D_n}\{\hat f_n(t)-\hat f_n^{(0)}(t)\}\,dt+O_p\left(n^{-1/3}\right),
\end{align*}
by which  we have reduced the log likelihood integrals on the shrinking neighborhood $D_n$.

We now proceed as in the proof of Theorem \ref{th:LR_current_status}.
We expand the logarithm in a neighborhood of the point $f_0(t_0)$. This yields:
\begin{align*}
&2n\int_{D_n}\log \frac{\hat f_n(t)}{\hat f_n^{(0)}(t)}\,d\F_n(t)
-2n\int_{D_n}\bigl(\hat f_n(t)-\hat f_n^{(0)}(t)\bigr)\,dt\\
&=2n\int_{D_n}\frac{\hat f_n(t)-f_0(t_0)}{f_0(t_0)}\,d\F_n(t)-2n\int_{D_n}\frac{\hat f_n^{(0)}(t)-f_0(t_0)}{f_0(t_0)}\,d\F_n(t)-2n\int_{D_n}\{\hat f_n(t)-\hat f_n^{(0)}(t)\}\,dt\\
&\qquad-n\int_{D_n}\frac{\bigl\{\hat f_n(t)-f_0(t_0)\bigr\}^2}{f_0(t_0)^2}\,d\F_n(t)
+n\int_{D_n}\frac{\bigl\{\hat f_n^{(0)}(t)-f_0(t)\bigr\}^2}{f_0(t_0)^2}\,d\F_n(t)+O_p\left(n^{-1/3}\right)\\
&=2n\int_{D_n}\frac{\hat f_n(t)-f_0(t_0)}{f_0(t_0)}\,d\F_n(t)-2n\int_{D_n}\frac{\hat f_n^{(0)}(t)-f_0(t_0)}{f_0(t_0)}\,d\F_n(t)-2n\int_{D_n}\{\hat f_n(t)-\hat f_n^{(0)}(t)\}\,dt\\
&\qquad-n\int_{D_n}\frac{\bigl\{\hat f_n(t)-f_0(t_0)\bigr\}^2}{f_0(t_0)}\,dt
+n\int_{D_n}\frac{\bigl\{\hat f_n^{(0)}(t)-f_0(t)\bigr\}^2}{f_0(t_0)}\,dt+O_p\left(n^{-1/3}\right)\\
\end{align*}
We now have:
\begin{align*}
&2n\int_{D_n}\frac{\hat f_n(t)-f_0(t_0)}{f_0(t_0)}\,d\F_n(t)-2n\int_{D_n}\{\hat f_n(t)-f_0(t_0)\}\,dt\\
&=2n\int_{D_n}\frac{\hat f_n(t)-f_0(t_0)}{f_0(t_0)}\hat f_n(t)\,dt-2n\int_{D_n}\{\hat f_n(t)-f_0(t_0)\}\,dt\\
&=2n\int_{D_n}\frac{\{\hat f_n(t)-f_0(t_0)\}^2}{f_0(t_0)}\,dt,
\end{align*}
and similarly get:
\begin{align*}
&2n\int_{D_n}\frac{\hat f_n^{(0)}(t)-f_0(t_0)}{f_0(t_0)}\,d\F_n(t)-2n\int_{D_n}\{\hat f_n^{(0)}(t)-f_0(t_0)\}\,dt\\
&=2n\int_{D_n}\frac{\{\hat f_n^{(0)}(t)-f_0(t_0)\}^2}{f_0(t_0)}\,dt,
\end{align*}
using $\hat f_n(t)=f_0(t_0)$ on the interval of constancy of $\hat f_n^{(0)}$, covering the point $t_0$.

So we can conclude:
\begin{align*}
&2n\int\log\frac{\hat f_n(t)}{\hat f_n^{(0)}(t)}\,d\F_n(t)\\
&=\frac{n}{f_0(t_0)}\int_{D_n}\left\{\{\hat f_n(t)-f_0(t_0)\}^2-\{\hat f_n^{(0)}(t)-f_0(t_0)\}^2\right\}\,dt+
O_p\left(n^{-1/3}\right)\\
&=\frac{n^{2/3}}{f_0(t_0)}\int_{n^{1/3}(a_n-t_0)}^{n^{1/3}(b_n-t_0)}\left\{\{\hat f_n(t_0+n^{-1/3}t)-f_0(t_0)\}^2-\{\hat f_n^{(0)}(t_0+n^{-1/3}t)-f_0(t_0)\}^2\right\}\,dt\\
&\qquad\qquad\qquad\qquad\qquad\qquad\qquad\qquad\qquad\qquad\qquad\qquad\qquad\qquad\qquad\qquad +O_p\left(n^{-1/3}\right),
\end{align*}
where $D_n=(a_n,b_n)$. 

Let $W$ be standard two-sided Brownian motion on $\R$, and let $\a=\sqrt{f_0(t_0)}$ and $\b=\frac12|f_0'(t_0)|$. The process
$$
t\mapsto \left(n^{1/3}\{\hat f_n(t_0+n^{-1/3}t)-f_0(t_0),n^{1/3}\{\hat f_n^{(0)}(t_0+n^{-1/3}t)-f_0(t_0)\}\right)
$$
converges on bounded intervals in the Skohorod topology to the process $(S_{\a,\b},S_{\a,\b}^{(0)})$ on $\R$, where $S_{\a,\b}$ is the slope of the concave majorant of the process
\begin{equation}
\label{ab_process}
t\mapsto X_{\a,\b}(t)\stackrel{\text{\small def}}=aW(t)-\b t^2,\qquad t\in\R,
\end{equation}
and where $S_{\a,\b}^{(0)}$ is defined by $S_{\a,\b}^-(t)\vee0$ for $t<0$, where $S_{\a,\b}^-$ is the slope of the process (\ref{ab_process}), restricted to the interval $(-\infty,0)$, and by $S_{\a,\b}^+(t)\wedge0$, where $S_{\a,\b}^+$ is the slope of the process (\ref{ab_process}), restricted to the interval $[0,\infty)$. The notation $X_{\a,\b}$ was introduced in \cite{mouli_jon:01}, p.\ 1706.

We now follow the Brownian scaling argument on p.\ 1724 of \cite{mouli_jon:01}. Let
$$
X(t)= X_{1,1}(t),\qquad t\in\R.
$$
Then
$$
X_{\a,\b}(t)\stackrel{{\cal D}}=\frac{\a^{4/3}}{\b^{1/3}}X\left((\b/\a)^{2/3}t\right),\qquad t\in\R.
$$
It follows that
\begin{align*}
\left(S_{\a,\b},S_{\a,\b}^{(0)}\right)\stackrel{{\cal D}}=\a^{2/3}\b^{1/3}\left(S_{1,1}((\b/\a)^{2/3}t),S^{(0)}_{1,1}((\b/\a)^{2/3}t)\right).
\end{align*}
So we get in the limit, noting that $S_{\a,\b}$ and $S_{\a,\b}^{(0)}$ only differ on a bounded interval,
\begin{align*}
&\frac1{\a^2}\int\left\{S_{\a,\b}(t)^2-S_{\a,\b}^{(0)}(t)^2\right\}\,dt
=(\b/\a)^{2/3}\int\left\{S_{1,1}((\b/\a)^{2/3}t)^2-S^{(0)}_{1,1}((\b/\a)^{2/3}t)^2\right\}\,dt\\
&=\int\left\{S_{1,1}(t)^2-S^{(0)}_{1,1}(t)^2\right\}\,dt.
\end{align*}

\end{proof}

\bibliographystyle{amsplain}
\bibliography{cupbook}

\end{document}

%% file: commands3.tex
\def\a{\alpha}
\def\b{\beta}
\def\R{\mathbb R}
\def\dd{\Delta}

\def\d{\delta}

\def\F{{\mathbb F}}
\def\G{{\mathbb G}}

\def\IK{I\!\!K}

\def\P{{\mathbb P}}

\def\l{\lambda}
\def\labda1{\lambda_1}
\def\labda2{\lambda_2}
\def\m{\mu}

\def\e{\varepsilon}
\def\f{\phi}
\def\t{\tau}

\def\s{\sigma}

\def\comment#1{\relax}

\def\=in{\mathop{\rm =}}

\newtheorem{theorem}{Theorem}[section]

\newtheorem{lemma}{Lemma}[section] 

\newtheorem{example}{Example}[section]
\newtheorem{remark}{Remark}[section]